\documentclass[12pt, a4paper]{amsart}
\usepackage[left=3cm, right=3cm, top=2.5cm, bottom=2cm]{geometry}
\usepackage{amsmath,amssymb,amsthm}

\usepackage{comment}
\usepackage{graphicx}
\usepackage{pstricks-add}
\usepackage{enumitem}
\usepackage{mathrsfs} 
\usepackage{xspace}

\usepackage[pagebackref=true]{hyperref}
\usepackage{cleveref}

\usepackage[colorinlistoftodos]{todonotes}
\usepackage{marginnote}
\makeatletter
\renewcommand{\@todonotes@drawMarginNoteWithLine}{%
\begin{tikzpicture}[remember picture, overlay, baseline=-0.75ex]%
    \node [coordinate] (inText) {};%
\end{tikzpicture}%
\marginnote[{
    \@todonotes@drawMarginNote%
    \@todonotes@drawLineToLeftMargin%
}]{
    \@todonotes@drawMarginNote%
    \@todonotes@drawLineToRightMargin%
}%
}
\makeatother

\usepackage{sidenotes}

\let\oldaddcontentsline\addcontentsline
\newcommand{\stoptocentries}{\renewcommand{\addcontentsline}[3]{}}
\newcommand{\starttocentries}{\let\addcontentsline\oldaddcontentsline}

\title[Type I and boundary representations of hyperbolic groups]{A type~I conjecture and boundary representations of hyperbolic groups}

\author[P.-E. Caprace]{Pierre-Emmanuel Caprace}
\thanks{P.-E.C. is a F.R.S.-FNRS senior research associate}
\address{UCLouvain, IRMP, Chemin du Cyclotron 2, bte L7.01.02, 1348 Louvain-la-Neuve, Belgique}
\email{pe.caprace@uclouvain.be}

\author[M. Kalantar]{Mehrdad Kalantar}
\thanks{MK is supported by a Simons Foundation Collaboration Grant (\# 713667)}
\address{University of Houston, USA}
\email{mkalantar@uh.edu}

\author[N. Monod]{Nicolas Monod}
\thanks{}
\address{EPFL, Switzerland}
\email{nicolas.monod@epfl.ch}

\newtheorem{thm}{Theorem}[section]

\newtheorem{prop}[thm]{Proposition}
\newtheorem{lem}[thm]{Lemma}
\newtheorem{cor}[thm]{Corollary}
\newtheorem{conj}[thm]{Conjecture}

\newtheorem{thmintro}{Theorem}

\newtheorem{conjintro}[thmintro]{Conjecture}
\newtheorem{corintro}[thmintro]{Corollary}

\theoremstyle{definition}

\newtheorem{rmk}[thm]{Remark}

\newcommand{\Zb}{\mathbf{Z}}

\newcommand{\Nb}{\mathbf{N}}

\newcommand{\Rb}{\mathbf{R}}
\newcommand{\Qb}{\mathbf{Q}}

\newcommand{\Aut}{\mathrm{Aut}}

\DeclareMathOperator{\Ker}{Ker}

\newcommand{\Sub}{\mathbf{Sub}}
\newcommand{\Isom}{\mathrm{Isom}}
\newcommand{\one}{\boldsymbol{1}}

\begin{document}

\begin{abstract}
We establish new results on the weak containment of quasi-regular and Koopman representations of a second countable locally compact group $G$ associated with non-singular  $G$-spaces. We deduce   that any two boundary representations of a hyperbolic locally compact group are weakly equivalent. We also show that non-amenable hyperbolic locally compact groups with a cocompact amenable subgroup are characterized by the property that any two proper length functions are homothetic up to an additive constant. Combining those results with the work of \L.~Garncarek on the irreducibility of boundary representations of discrete hyperbolic groups, we deduce that a type~I hyperbolic group with a cocompact lattice contains a cocompact amenable subgroup. Specializing to groups acting on trees, we answer  a question of C.~Houdayer and S.~Raum.  
\end{abstract}

\maketitle

\section{Introduction}

\stoptocentries

\subsection{Type~I groups}
According to L.~Auslander and C.~Moore~\cite[p.~1]{AuslanderMoore},

\begin{flushright}
\begin{minipage}[t]{0.90\linewidth}
``\itshape one of the fundamental questions one can raise about any locally compact group or C*-algebra is that of determining when it is type~I\upshape.''
\end{minipage}
\end{flushright}

\noindent
The notion of type I, hailing from the very origins of operator algebras and representation theory~\cite{Murray-vonNeumannI}, can be seen as a rigorous way to define the class of groups for which unitary representations can be classified in any meaningful manner. That notion is of fundamental importance, although the definition may seem technical at first sight; we refer to the recent book~\cite{BekkaHarpe} for a detailed account and a review of known result. Let us merely mention here that, by a celebrated result of E.~Thoma~\cite{Thoma64,Thoma}, \emph{a discrete group is type~I if and only if it is virtually abelian}.

In the non-discrete case, the current state of the art is not nearly as complete, despite numerous results ensuring that various important families of groups are type~I and thus showing that the type~I class is much richer than in the discrete case. What is completely lacking, in contrast to Thoma's theorem, is a definite structural consequence of type~I. We venture the following.

\begin{conjintro}\label{conj:main}
Every second countable locally compact group of type~I admits a cocompact amenable subgroup. 
\end{conjintro}

The structural property contemplated in this conjecture is a strong one. For instance, hyperbolic unimodular groups with a cocompact amenable subgroup are described in detail in~\cite{CCMT}. In the wide setting of non-positively curved (CAT(0)) groups, a classification is obtained in~\cite{CM-bib} using also~\cite{CM-structure, CM-fixed}. These descriptions are in terms of semi-simple groups (Lie or non-Archimedean) and tree automorphism groups, and it turns out that those classes include most known examples of type~I groups with a trivial amenable radical.

(The study of type~I among amenable groups is also interesting and very incomplete, but~\Cref{conj:main} is obviously irrelevant there.)

A more partial motivation for \Cref{conj:main} is that \emph{Gelfand pairs} $(G,K)$, which are characterised by a very tame $K$-spherical dual for a compact subgroup $K<G$, admit a cocompact amenable subgroup $P<G$, indeed even an ``Iwasawa decomposition'' $G=K P$~\cite{Monod_Gelfand}.

\Cref{conj:main}  is further discussed in Section~\ref{sec:C*-simple}. We show that this conjecture would follow from the simplicity of the $C^*$-algebra $C^*_{\lambda_{G/H}}(G)$, where $H$ is the stabilizer of a generic point in the Furstenberg boundary of $G$ (see \Cref{prop:conj=>conj} and \Cref{rem:tdlc-reduction}).

The bulk of this article is devoted to a completely different line of attack towards \Cref{conj:main}, in a geometric set-up that fosters a rich interplay between the dynamical and geometrical aspects of boundary theory. Namely, one of the conclusions of this paper is that \Cref{conj:main} holds for all  {hyperbolic} locally compact groups which admit a uniform lattice. 

\begin{thmintro}\label{thm:main}
Let $G$ be a hyperbolic locally compact group admitting a uniform lattice.

If $G$ is type~I, then $G$ has a cocompact amenable subgroup. 
\end{thmintro}

As hinted to above, this result leads to a rather precise description of $G$ thanks to~\cite[Theorem~D]{CCMT}:

\begin{corintro}\label{cor:hyper-coco-amen}
Let $G$ be a hyperbolic locally compact group admitting a uniform lattice. Recall that $G$ admits a unique maximal compact normal subgroup $W$.

If $G$ is type~I, then the quotient group $G/W$ satisfies exactly one of the following descriptions:

\begin{enumerate}[label=(\roman*)]
\item $G/W$ is the group of isometries of a rank one symmetric space of non-compact type, or its identity component, which has index at most~$2$.\\ In particular $G/W$ is a simple Lie group of rank one.\label{it:hyper-coco-amen:Lie}
\item $G/W$ is a closed subgroup of the automorphism group of a locally finite non-elementary tree $T$, acting without inversions and with exactly two orbits of vertices, and acting $2$-transitively on the set $\partial T$ of ends.\label{it:hyper-coco-amen:tree}
\item $G/W$ is trivial or isomorphic to $\Zb$, $\Rb$, $\Zb \rtimes \{\pm 1\}$ or $\Rb \rtimes \{\pm 1\}$. \label{it:hyper-coco-amen:elem}
\end{enumerate}
In all cases, it follows that $G$ has an Iwasawa decomposition $G = K P$, where $K$ is a compact subgroup and $P$ an amenable closed subgroup. 
\end{corintro}

It is tempting to believe that \Cref{cor:hyper-coco-amen} could lead to a necessary and sufficient structural characterisation of type~I among hyperbolic locally compact group admitting a uniform lattice. Indeed, every group as in~\ref{it:hyper-coco-amen:Lie} or~\ref{it:hyper-coco-amen:elem} is known to be type~I; as for case~\ref{it:hyper-coco-amen:tree}, these groups are conjectured to be type~I since~\cite{Nebbia} (C.~Nebbia's conjecture in loc.~cit. is formulated for regular trees, and predicts the formally stronger statement that a group as in (ii) is CCR).

As we shall discuss below (see Remark~\ref{rem:Garn}), the assumption on the existence of a uniform lattice can probably be relaxed with further work. Two points should be emphasized in that relation:

On the one hand, this assumption incidentally ensures that $G$ is unimodular,  
which rules out \emph{amenable} non-elementary  hyperbolic groups~\cite[Theorem~7.3]{CCMT}. Those  can indeed fail to be type~I and rule out the most naive converse to \Cref{conj:main}, see \Cref{prop:non-type-I} for an example.

On the other hand, if as we suspect we can replace the assumption on uniform lattices by the weaker unimodularity assumption, then this stronger version of \Cref{cor:hyper-coco-amen} would imply a posteriori that $G$ \emph{does} contain a uniform lattice. Indeed, this follows from the Borel--Harish-Chandra theorem~\cite{Borel-Harish-Chandra} in case~\ref{it:hyper-coco-amen:Lie}, from the Bass--Kulkarni theorem~\cite{BaKu} in case~\ref{it:hyper-coco-amen:tree}, and is clear in case~\ref{it:hyper-coco-amen:elem}.


One of the upshots of \Cref{cor:hyper-coco-amen} is that the totally disconnected case is reduced to the setting of tree automorphism groups. Taking a step back, we note that the class of compactly generated closed subgroups $G$ of the automorphism group of any locally finite tree $T$ provides a broad natural source of examples of hyperbolic locally compact groups. Indeed, the group $G$ then acts cocompactly on the smallest $G$-invariant subtree $T'$ of $T$ (see~\cite[Lemma~2.4]{CDM}), so that $G$ is quasi-isometric to $T'$. In particular $G$ is hyperbolic. By specializing  Theorem~\ref{thm:main} and Corollary~\ref{cor:hyper-coco-amen} to this setting, we obtain the following.  

\begin{corintro}\label{cor:HR}
Let $T$ be a locally finite tree and $G \leq \Aut(T)$ be a closed non-amenable subgroup acting minimally on $T$. If $G$ is type~I, then the $G$-action on the set of ends $\partial T$ is $2$-transitive. 
\end{corintro}

This strengthens  Theorem~A from~\cite{HoudayerRaum}, establishing the weaker property that the $G$-action on $T$ is \emph{locally $2$-transitive}, i.e. the stabilizer of every vertex of valency~$\geq 3$ is $2$-transitive on the set of incident edges. 
Corollary~\ref{cor:HR} contributes to the characterization of the type~I property among closed subgroups of the automorphism group of a tree by solving the second part of Problem~1 in~\cite{HoudayerRaum}. It should also be noted that C.~Houdayer and S.~Raum establish a stronger result than their Theorem~A mentionned above (see~\cite[Theorem~C]{HoudayerRaum}), which establishes the same conclusion under the  hypothesis that the group von Neumann algebra $L(G)$ is amenable, which is formally weaker than the type~I condition. 

\subsection{Boundary representations}
Glimm's theorem~\cite{Glimm} characterizes the type~I property for $G$ by the fact that any two weakly equivalent irreducible unitary representations of $G$ are equivalent. For this reason, a substantial part of this paper is a contribution to the unitary representation theory of hyperbolic locally compact groups. An important source of unitary representations of such a group $G$ is provided by the so-called \textbf{boundary representations}, which are the Koopman representations $\kappa_\nu$ associated with a quasi-invariant probability measure $\nu$ supported on the Gromov boundary  $\partial G$. We shall prove the following. 

\begin{thmintro}\label{thm:weak-equiv-bd-rep}
Any two boundary representations of a non-amenable hyperbolic locally compact group are weakly equivalent.
\end{thmintro}

For a detailed account of notion of \emph{weak containment} and  \emph{weak equivalence} of unitary representations, we refer to~\cite[Appendix~F]{BHV}. Let us simply recall here that two unitary representations $\pi_1, \pi_2$ of a locally compact group $G$ are weakly equivalent if their $C^*$-kernel coincide; in other words the respective representations of $C^*(G)$ corresponding to $\pi_1, \pi_2$ have the same kernel. For a discrete non-elementary hyperbolic group $G$, it is well-known that every boundary representation of $G$ is weakly contained in the regular representation (see~\cite[Theorem~5.1]{Adams_bd} and~\cite{Kuhn}). Using that   the kernel of the $G$-action on $\partial G$ is the amenable radical $R(G)$, which is finite, the assertion of Theorem~\ref{thm:weak-equiv-bd-rep} for a \emph{discrete} hyperbolic group directly follows from the fact that $G/R(G)$ is \textbf{$C^*$-simple}, i.e. its the reduced $C^*$-algebra  is simple, see~\cite{Harpe88} and~\cite[Theorem~6.5]{BKKO}. For \emph{non-discrete} hyperbolic groups, the latter property fails: indeed, every  rank~one connected simple Lie group with finite center is hyperbolic, but no such group is $C^*$-simple  (see~\cite[App.~G]{Harpe-C*}).

The proof of \Cref{thm:main} combines \Cref{thm:weak-equiv-bd-rep} with several other ingredients that we now proceed to describe. In trying to apply Glimm's characterization of the type~I condition, it is obviously useful to have a large supply of irreducible representations of $G$ at one's disposal, together with a good understanding of their classification up to equivalence. In the setting of \Cref{thm:main}, these representations will be the boundary representations associated with \emph{Patterson--Sullivan measures}.

Specifically, consider a locally compact hyperbolic group $G$. For the purpose of \Cref{thm:main}, we can readily reduce ourselves to the case where $G$ is totally disconnected. In particular, it admits a Cayley--Abels graph $X$ on the vertex set $G/U$, where $U$ is any given compact open subgroup. More precisely, the edge structure of the graph $X$ depends on the choice of a compact generating set for $G$, or equivalently of a suitable word metric on $G$; we will crucially use this freedom of choice.

Since $X$ is a locally finite hyperbolic graph, its Gromov boundary is a compact $G$-space and supports a canonical family of Patterson--Sullivan measures. For the time being, recall simply that any choice of a base-point in $X$ determines a measure $\nu$ on $\partial X \cong \partial G$ and that any other point gives an equivalent measure; in particular, $\nu$ is quasi-invariant under $G$ and defines a Koopman representation $\kappa_\nu$ of $G$ on $L^2(\partial G, \nu)$, that we call a \textbf{PS-representation}.

Our strategy consists in first using \Cref{thm:weak-equiv-bd-rep} to apply Glimm's criterion to these representations and then establishing a number of geometric consequences culminating in the fact that $G$ acts transitively on its boundary $\partial G$. To this end, we rely on the groundbreaking work of \L.~Garncarek~\cite{Garncarek}, extending earlier results of  Bader--Muchnik~\cite{BadMuch}, and solving an important case of the conjecture that they formulated in loc.~cit. The fact that Garncarek's results are stated and proved for \emph{discrete} hyperbolic groups explains our auxiliary hypothesis that $G$ admits a uniform lattice. On the one hand, Garncarek proves that PS-representations are all irreducible. On the other hand, he shows that two such representations are equivalent if and only if the underlying word metrics are \textbf{roughly similar}, which means that they are homothetic up to an additive constant. Applying these results to the lattice yields them a fortiori for $G$, since the representations are already defined on $G$.

In summary, relying on Glimm's and Garncarek's theorems, our strategy consists of the following two steps.

The first one, of analytic flavour, is to prove \Cref{thm:weak-equiv-bd-rep}.  It implies that all PS-representations (associated to any word metric) are weakly equivalent. We shall establish such a weak equivalence in a more general context for rather more general groups $G$, as stated in \Cref{cor:amenable-actions} below.

The second step, of geometric flavour, is to establish that $G$ admits indeed a cocompact amenable subgroup when any two word metrics on $G$ are roughly similar. This requires new results on hyperbolic groups, entering \Cref{thm:coctamen} below.

\subsection{Weak containment of Koopman and quasi-regular representations}
The first remaining step for the proof of Theorem~\ref{thm:weak-equiv-bd-rep} relies on new results on weak containment of unitary representations in a broad context of locally compact transformation groups, regardless of any hyperbolicity assumption. In order to present their statements, we recall that a locally compact group $H$ is called \textbf{regionally elliptic} if every compact subset of $H$ is contained in a compact subgroup. If $H$ is $\sigma$-compact, this is equivalent to requiring that $H$ is a countable ascending union of compact subgroups. A \textbf{tdlc group} is a locally compact group which is totally disconnected.

\begin{thmintro}\label{thm:weak-containment}
Let $G$ be a  second countable tdlc group and $(X, \nu)$ be a standard probability space endowed with a measurable $G$-action, such that $\nu$ is quasi-invariant under $G$. Let $\kappa$ be the Koopman unitary representation of $G$ on $L^2(X, \nu)$. 

Then there is a co-null set $Y\subset X$ such that for every $x\in Y$ with regionally elliptic stabilizer $G_x$, the quasi-regular representation $\lambda_{G/G_x}$ is weakly contained in the Koopman representation $\kappa$.
\end{thmintro}

Let us point out that for a discrete group, every Koopman representation weakly contains the quasi-regular representation associated with almost every point stabilizer  (see~\cite[Proposition~7]{DuGr}). We do not know whether this result holds in the non-discrete case; Theorem~\ref{thm:weak-containment} provides a  special case where this is indeed true.

The relevant weak containment in the opposite direction is established by the following statement, which crucially relies on the work of C.~Anantharaman-Delaroche \cite{AD03}.

\begin{thmintro}\label{thm:Ananth}
Let $G$ be a second countable locally compact group, and let $X$ be a 
minimal compact $G$-space. 
Let $\nu$ be a $G$-quasi-invariant Radon probability measure on $X$ such that $L^2(X, \nu)$ is separable. Assume that the $G$-action on $(X, \nu)$ is amenable in the sense of Zimmer, and let $X_1 \subseteq X$ denotes the conull subset consisting of those $x \in X$ such that $G_x$ is amenable. Then the following assertions hold. 
\begin{enumerate}[label=(\roman*)]
\item For  all $x, y \in X_1$, the quasi-regular representations $\lambda_{G/G_x}$ and $\lambda_{G/G_y}$ are weakly equivalent. 

\item For  all $x \in X_1$, the quasi-regular representation $\lambda_{G/G_x}$  weakly contains the Koopman representation $\kappa_{\nu}$.
\end{enumerate}
\end{thmintro}

We emphasize that, if the $G$-action on $X$ is topologically amenable, then $X_1 = X$.
%

By combining a topological version of Theorem~\ref{thm:weak-containment}, recorded as Theorem~\ref{thm:weak-cont-topol} below, with Theorem~\ref{thm:Ananth}, we obtain the following  consequence. 

%

\begin{corintro}\label{cor:amenable-actions}
Let $G$ be a  second countable tdlc group and $X$ be a minimal compact  $G$-space equipped with a $G$-quasi-invariant Radon probability measure $\nu$ such that $L^2(X, \nu)$  is separable. Suppose that for some $x \in X$, the stabilizer $G_x$ is regionally elliptic. 

If the $G$-action on $(X, \nu)$ is amenable in Zimmer's sense, then the Koopman representation $\kappa$ of $G$ on $L^2(X, \nu)$ is weakly equivalent to the quasi-regular representation $\lambda_{G/G_y}$ for any $y \in X$ such that $G_y$ is amenable.
\end{corintro}

Thus in particular, in the setting of Corollary~\ref{cor:amenable-actions}, the weak equivalence class of the Koopman representation $\kappa$ is independent of $\nu$. 

The regionally elliptic hypothesis appearing in Theorem~\ref{thm:weak-containment} and Corollary~\ref{cor:amenable-actions} may seem rather restrictive (although the special case where $G_x = \langle e \rangle$ is interesting in its own). 
Nonetheless,  it turns out that the results of this section are general enough to be applied to the action of a hyperbolic locally compact group $G$ on its Gromov boundary $X = \partial G$: that action is indeed topologically amenable (see~\cite[Theorem~6.8]{Adams} and~\cite{Kaimanovich}). Moreover, we establish new results on the algebraic structure of amenable subgroups of $G$ which  ensure that  the regionally elliptic hypothesis of Corollary~\ref{cor:amenable-actions} is automatically satisfied in the context of Theorem~\ref{thm:weak-equiv-bd-rep} (see Section~\ref{sec:hyp-rel-amen} below), using a reduction  to the tdlc case that relies on~\cite{CCMT}.

\subsection{Characterizing non-amenable hyperbolic groups with a cocompact amenable subgroup}

As mentioned above, the proof of \Cref{thm:main} requires identifying the non-amenable hyperbolic groups with a cocompact amenable subgroup with those hyperbolic groups  for which any two word metrics are roughly similar. This is ensured by \Cref{thm:coctamen},   which supplements~\cite[Theorem~D and Theorem~8.1]{CCMT}. Given a point $\xi$ in the Gromov boundary of a hyperbolic locally compact group $G$, we denote by $G_\xi^0$ the kernel of the Busemann homomorphism $\beta_\xi \colon G_\xi \to \mathbf R$ (see Section~\ref{sec:hyp-rel-amen} below). 

\addtocounter{thmintro}{1}
\begin{thmintro}\label{thm:coctamen}
Let $G$ be a non-amenable hyperbolic locally compact group and $(X, d)$ be a proper geodesic metric space on which $G$ acts continuously, properly and cocompactly by isometries.

The following conditions are equivalent.

\begin{enumerate}[label=(\roman*)]
\item\label{it:coco} 
$G$ has a cocompact amenable subgroup. 

\item\label{it:bd-2-tr} The $G$-action on the Gromov boundary $\partial X$ is $2$-transitive. 

\item\label{it:all-ballistic} 
For all $\xi \in \partial X$, we have $G_\xi^0 \neq G_\xi$. 

\item\label{it:generic-ballistic} 
For some continuity point $\eta \in \partial X$ of the stabilizer map, we have $G_\eta^0 \neq G_\eta$. 

\item\label{it:bdd-tsl} 
There is a constant $K$ such that, for every hyperbolic element $\gamma \in G$, there exists a hyperbolic element $\gamma' \in G$ with asymptotic displacement length~$|\gamma'|_\infty \leq K$ such that $\gamma$ and $\gamma'$ share the same pair of fixed points in $\partial X$.

\item\label{it:metric} 
For any word metric $d'$ on $G$ with respect to a compact generating set, each orbit map $G\to X$ is a rough similarity.
\end{enumerate}
\end{thmintro}

The equivalence between~\ref{it:coco} and~\ref{it:bd-2-tr} is taken from~\cite{CCMT}. A special instance of the equivalence between~\ref{it:bd-2-tr} and~\ref{it:all-ballistic} has been observed for specific families of groups of tree automorphisms by C.~Ciobotaru in her PhD thesis~\cite[Proposition~2.2.11]{Cio}. The proof of Theorem~\ref{thm:main} uses   the implications~\ref{it:generic-ballistic} $\Rightarrow$~\ref{it:coco},~\ref{it:bdd-tsl} $\Rightarrow$~\ref{it:coco} and~\ref{it:metric} $\Rightarrow$~\ref{it:coco}, which are all new.

\subsection{Epilogue on boundary representations}
We finish by recording another result, established along the way, that has its own interest. We recall that a $C^*$-algebra $A$ is called \textbf{CCR} if $\pi(A)$ consists of compact operators for every irreducible representation $\pi$ of $A$ (\cite[Definition~4.2.1]{Dixmier}). $A$ is called \textbf{GCR} if every non-zero quotient of $A$ contains a non-zero CCR closed two-sided ideal (\cite[Definition~4.3.1]{Dixmier}).

We say a unitary representation $\pi$ of a locally compact group $G$ is {CCR} (resp. {GCR}) if the C*-algebra $C^*_\pi(G) := \pi(C^*(G))$ is {CCR} (resp. {GCR}). 
In particular, if $\pi$ is an irreducible  CCR representation,  then  $C^*_\pi(G)$ consists of  compact operators, but this need not be the case if $\pi$ is not irreducible (indeed, a normal operator need not be compact).  It is thus important to underline  that the boundary representation $\kappa$ in the following theorem  is arbitrary, and need not be irreducible a priori.

\begin{thmintro}\label{thm:GCR}
Let $G$ be a non-amenable hyperbolic locally compact group admitting a uniform lattice. For any boundary representation $\kappa$ of $G$, the following assertions are equivalent. 

\begin{enumerate}[label=(\roman*)]
\item\label{it:CCR-ideal}  $C^*_\kappa(G)$ contains a non-zero CCR closed two-sided  ideal. 

\item\label{it:GCR} $\kappa$ is GCR. 

\item\label{it:CCR}   $\kappa$ is CCR.

\item\label{it:all-compact-op} $C^*_\kappa(G)$ entirely consists of  compact operators.

\item\label{it:amen-coco} $G$ has a cocompact amenable subgroup. 
\end{enumerate}
\end{thmintro}

The implications~\ref{it:all-compact-op}~$\Rightarrow$~\ref{it:CCR}~$\Rightarrow$~\ref{it:GCR}~$\Rightarrow$~\ref{it:CCR-ideal} in Theorem~\ref{thm:GCR} are tautological, while the implication~\ref{it:amen-coco}~$\Rightarrow$~\ref{it:all-compact-op} follows from the fact that if $G$ has a cocompact amenable subgroup $P$, then $\partial G$ can naturally be identified with $G/P$ so that $\kappa$ becomes equivalent to the quasi-regular representation $\lambda_{G/P}$. That the quasi-regular representation defined by a cocompact subgroup satisfies the condition~\ref{it:all-compact-op} follows by general principles (compare Proposition~\ref{prop:CCR-coco}).

The key implication is~\ref{it:CCR-ideal}~$\Rightarrow$~\ref{it:amen-coco}.
The formally weaker implication~\ref{it:all-compact-op}~$\Rightarrow$~\ref{it:amen-coco} is much more straightforward, and   can be established without requiring that $G$ has a cocompact lattice by invoking  Propositions~\ref{prop:RelAmen-action} and~\ref{prop:transitive-boundary} below (see also~\cite{Nebbia} in the special case of groups acting on trees).

\tableofcontents

\starttocentries

\section{Preliminaries}\label{sec:prem}
In this section we gather some general facts which we will use in proofs of our results.

\subsection{Koopman unitary representations}
Throughout the paper, by a representation of a locally compact group $G$ we always mean a continuous unitary representation. The most commonly used representations in this work are the Koopman unitary representations $\kappa_\nu$ associated to measurable actions of locally compact groups $G$ on $G$-quasi-invariant $\sigma$-finite measure spaces $(X, \nu)$. Recall that $\kappa_\nu$ is defined by 
\[
(\kappa_\nu(g) \xi) (x) := \sqrt{\frac{dg\nu}{d\nu}(x)}\, \xi(g^{-1}x)
\]
for all $g\in G$, $\xi\in L^2(X, \nu)$ and $\nu$-a.e. $x\in X$. 

Some care should be taken with regards to the continuity of the action, which is not guaranteed in this generality (see~\cite[Remark A.6.3]{BHV}).
The representation $\kappa_\nu$ is a continuous, for instance, when $G$ is $\sigma$-compact and $L^2(X, \nu)$ is separable (\cite[Proposition A.6.1]{BHV} and~\cite[Theorem 2]{SegVN50}).
Except in some places in Section~\ref{sec:C*-simple}, we always work with continuous actions of $G$ on metrizable locally compact spaces $X$. When $\nu$ is a Radon measure on $X$, then $L^2(X, \nu)$ is separable. Also, note that all locally compact hyperbolic groups are $\sigma$-compact. In particular, the continuity issue will only be relevant in some places in the last section, where we consider general boundary actions. However, in that case, considering those measures $\nu$ whose $L^2$-spaces are separable will suffice for our purposes, and therefore again we will not have any continuity issue for the Koopman representations.

The following basic facts are well-known. We record them for easy reference in few places later in the paper. 

\begin{lem} \label{lem:equiv-meas->equiv-koop}
Let $G$ be a $\sigma$-compact locally compact group and $X$ be a locally compact $G$-space. If $\nu$ and $\nu'$ are two equivalent $\sigma$-finite Radon measures on $X$ which are quasi-invariant under $G$, then the Koopman representations of $G$ on $L^2(X, \nu)$ and $L^2(X, \nu')$ are unitary equivalent.
\end{lem}
\begin{proof}
It is straightforward to see that the map $T: L^2(X, \nu)\to L^2(X, \nu')$ defined by $T\xi = \sqrt{\frac{d\nu}{d\nu'}}\, \xi$ is a unitary that intertwines $\kappa_\nu$ and $\kappa_{\nu'}$.
\end{proof}

\begin{lem} \label{lem:conv->equiv-meas}
Let $G$ be a locally compact group and $X$ be a locally compact $G$-space. Let $\nu$ be a Borel probability measure on $X$ which is quasi-invariant under $G$. Then for any Borel measure $\mu$ on $G$ we have $\mu*\nu\sim \nu$.
\end{lem}
\begin{proof}
Let $Y$ be a Borel subset of $X$. If $\nu(Y) = 0$, then $\nu(gY) = 0$ for all $g\in G$ by quasi-invariance, and therefore $\mu*\nu(Y) = \int_G \nu(g^{-1}Y)\, d\mu(g) = 0$.

Conversely, assume $\mu*\nu(Y) = 0$. Then $\nu(g^{-1}Y) = 0$ for $\mu$-a.e. $g\in G$. Hence, $\nu(Y) = 0$ by quasi-invariance.
\end{proof}

We recall from~\cite[\S13.7]{Dixmier} that  a \textbf{positive definite measure} on a locally compact group $G$ is a measure $\mu$ on $G$ such that $\int_G f * \tilde f \, d\mu \geq 0$ for all $f \in C_c(G)$, where $C_c(G)$ denotes the set of continuous compactly supported complex valued functions on $G$ and $\tilde f(g) = \overline{f(g^{-1})}$. (The convolution is with respect to a left Haar measure.) Every such measure $\mu$ defines a unitary representation $\pi_\mu$ of $G$ constructed as follows. The formula  $\langle f, g\rangle_\mu =  \int_G   \tilde g * f\,d\mu $ turns $C_c(G)$ into a pre-Hilbert space; we define the Hilbert space $\mathscr H_\mu$ to be its separated completion. The representation $\pi_\mu$ is induced by the representation $s$ on $C_c(G)$ defined by $(s(g)f)(x) = f(g^{-1}x) \Delta_G(g)^{1/2}$, where $\Delta_G$ is the modular function of $G$.

For example, for the Dirac mass $\delta_e$ at the neutral element, we have $\pi_{\delta_e} \cong \lambda_G$, whereas the representation associated with a right Haar measure on $G$ is the trivial representation. More generally, we have the following result of Blattner. 

\begin{thm}\label{thm:Blattner}
Let $G$ be locally compact group and $H \leq  G$ be a closed subgroup. Let $\alpha_H$ be a left Haar measure on $H$, viewed as a measure on $G$, and define a mesure $\mu$ on $G$ be setting $d\mu = \sqrt{\frac{\Delta_G}{\Delta_H}} d \alpha_H$, where $\Delta_G$ and $\Delta_H$ are the modular functions of $G$ and $H$ respectively. Then $\mu$ is positive definite and $\pi_\mu$ is equivalent to  the quasi-regular representation $\lambda_{G/H}:=\mathrm{Ind}_H^G(\mathbf 1)$. 
\end{thm}
\begin{proof}
We refer to~\cite[Theorem~1]{Blattner1963}; note that Blattner uses the opposite convention from Dixmier regarding convolution. 
\end{proof}

We endow the set of measures on $G$ with the \textbf{vague topology}, which is the topology of pointwise convergence on $C_c(G)$ (see~\cite[Chapter~III, \S1, no.~9]{Bourbaki_Int_Ch1-4}). In particular,  a sequence $(\mu_n)_n$ of measures on $G$ \textbf{vaguely converges} to a measure  $\mu$ if for every function $f  \in C_c(G)$, the sequence $(\int_G f d\mu_n)_n$ converges to $\int_G f d\mu$.

\begin{lem} \label{lem:VagueConv}
Let $G$ be a lcsc group and $\mathscr M$ be a set of measures on $G$. Then $\mathscr M$ is relatively compact for the vague topology if and only if for each compact subset $K \subseteq G$, there is a constant $M_K$ such that $|\mu|(K) \leq M_K$ for all $\mu \in \mathscr M$. 
\end{lem}

\begin{proof}
We refer to Proposition~15 in~\cite[Chapter~III, \S1, no.~9]{Bourbaki_Int_Ch1-4}.
\end{proof}

The following continuity principle is of fundamental importance for Theorem~\ref{thm:weak-containment}. It is related to Fell's continuity theorem recalled in Section~\ref{sec:Fell} below.

\begin{lem}\label{lem:weak-limit-posdef}
Let $G$ be a lcsc group and $(\mu_n)$ a sequence of positive definite measures on $G$. If $\mu_n \to \mu$ vaguely, then $\mu$ is also positive definite and $\pi_{\mu_n}\to \pi_\mu$ in Fell's topology.
\end{lem}

\begin{proof}
The fact that $\mu$ is also positive definite is obvious. The definition of Fell convergence needs only to be checked on a dense subset of vectors (see e.g.~\cite[Lemma~F.1.3]{BHV}), so we check it using the canonical image of $C_c(G)$ in the Hilbert spaces ${\mathscr H}_{\pi_{\mu_n}}$ and ${\mathscr H}_{\pi_{\mu}}$ respectively. Explicitly, given $f\in C_c(G)$ and a compact subset $Q\subset G$, it suffices to show that the matrix coefficient function
\[
g \longmapsto \left\langle f, \pi_{\mu_n}(g) f \right\rangle_{\mu_n} =  \mu_n(  (s(g) f)^{\sim} * f)
\]
converges to $\mu(  (s(g) f)^{\sim} *  f)$ uniformly over $g\in Q$. The vague convergence assumption implies that this holds pointwise for each $g$. However, the set $\{ (s(g) f) *  f : g\in Q\}$ is compact in $C_c(G)$ for the topology of uniform convergence on compact subsets because the orbital map $g\mapsto s(g) f$ is continuous. So we only need to recall that the vague convergence of a \emph{sequence} is actually uniform on compact subsets of $C_c(G)$, see Proposition~17(ii) in~\cite[Chapter~III, \S1, no.~10]{Bourbaki_Int_Ch1-4}.
\end{proof}

We shall also need the following basic fact. 

\begin{lem}\label{lem:stabilizer}
Let $G$ be a locally compact group and $X$ be a locally compact $G$-space. For each $x \in X$ and every compact subset $K \subset G$ with $K \cap G_x = \varnothing$, there exists an open neighbourhood $\alpha$ of $x$ such that $g\alpha \cap \alpha = \varnothing$ for all $g \in K$. 
\end{lem}
\begin{proof}
Suppose it is not the case. Then for each open neighbourhood  $\alpha$ of $x$, there exists $g_\alpha \in K$ and $x_\alpha \in \alpha$ with $g_\alpha(x_\alpha) \in \alpha$. The compactness of $K$ ensures that  $(g_\alpha)$ subconverges to an element $g \in K$ with $g(x) = x$, contradicting the hypothesis that $K \cap G_x = \varnothing$. 
\end{proof}

\subsection{URS and Fell's continuity theorem}\label{sec:Fell}
We recall Fell's continuity theorem:

\begin{thm}[Fell 1964]\label{thm:Fell}
The quasi-regular representation $\lambda_{G/H}$, viewed as map from the Chabauty space of closed subgroups $H<G$ of a given locally compact group $G$ to the Fell space of equivalence classes of unitary $G$-representations, is continuous. 
\end{thm}

\begin{proof}
This is contined in Theorem~4.2 of~\cite{Fell1964}.
\end{proof}

The relation with Lemma~\ref{lem:weak-limit-posdef} is that one can embed the Chabauty space into the space of Radon measures on $G$ by assigning to $H<G$ a suitably normalized Haar measure on $H$, viewed as a measure on $G$. When both modular functions $\Delta_G $ and $\Delta_H$ are trivial on $H$, this directly gives a proof of Theorem~\ref{thm:Fell}; we spell out the arugment for completeness:

\begin{prop}\label{prop:Fell-continuity}
Let $G$ be a  lcsc group and $(H_n)$ be a sequence of unimodular closed subgroups, all contained in $\Ker(\Delta_G)$, that converges to $H \leq G$ in the Chabauty topology. Then $\lambda_{G/H_n}$ converges to $\lambda_{G/H}$ in Fell's topology.
\end{prop}

We note that the assumptions on the modular functions are satisfied in the setting of Theorem~\ref{thm:weak-containment} since it is concerned with regionally elliptic subgroups $H$. Accordingly, the proof of that theorem will invoke Lemma~\ref{lem:weak-limit-posdef} rather than the full generality of Theorem~\ref{thm:Fell}.

\begin{proof}[Proof of Proposition~\ref{prop:Fell-continuity}]
Let $V$ be a compact identity neighbourhood in $G$. For each closed subgroup $J \leq G$, we fix the left Haar measure $\alpha_J$ on $J$ such that $\alpha_J(J \cap V) = 1$. We view $\alpha_J$ as a measure defined on $G$ and supported on $J$. 
The map $J \mapsto \alpha_J$ defines a homeomorphism of the Chabauty space $\Sub(G)$ onto its image, which is endowed with the vague topology, see~\cite[Chapter~VIII, \S 3 and \S 6]{Bourbaki}. Moreover, the set of unimodular closed subgroups is Chabauty closed by~\cite[Chapter~VIII, \S 3, Theorem~1]{Bourbaki}. It follows that $H$ is unimodular. 

For each $n$ define the measure $\mu_n$ by $d\mu_n = \sqrt{\Delta_G} d\alpha_{H_n}$, and define $\mu$ by $d\mu = \sqrt{\Delta_G} d\alpha_{H}$, where $\Delta_G$ is the modular function of $G$. Since $H_n \leq \Ker(\Delta_G)$ by hypothesis, we have $H \leq  \Ker(\Delta_G)$, and it follows $\mu_n=\alpha_{H_n}$, and $\mu=\alpha_{H}$. Therefore, we deduce from  Theorem~\ref{thm:Blattner} that $\pi_{\alpha_{H_n}}$ (resp. $\pi_{\alpha_H}$) is equivalent to $\lambda_{G/H_n}$ (resp. $\lambda_{G/H}$). By hypothesis, we know that $\alpha_{H_n}$ vaguely converges to $\alpha_H$. Hence, the required conclusion follows from Lemma~\ref{lem:weak-limit-posdef}.
\end{proof}

Let us now consider a locally compact group $G$. A \textbf{uniformly recurrent subgroup} (or \textbf{URS}) is a minimal $G$-invariant closed subset of $\Sub(G)$. 
We recall from~\cite{GlWe} that  every minimal compact $G$-space $X$ yields a URS called the \textbf{stabilizer URS}, defined as the unique URS contained in the closure of the set of stabilizers $\{G_x \mid x \in X\}$. It is denoted by   $\mathcal{ST}_G(X)$. If $\Sub(G)$ is metrizable (which is automatic if $G$ is second countable), then   a classical semi-continuity argument ensures that $X$ contains a dense $G_\delta$-set of points $x$  such that the stabilizer map $x \mapsto G_x$ is continuous (see~\cite[Theorem~VII]{Kuratowski1928}). The set of those continuity points is denoted by $X_0$. For $x \in X_0$, we have $G_x \in \mathcal{ST}_G(X)$. 

By Fell's continuity theorem, if $\mathcal Y$ is a URS of $G$, then for any two $Y_1, Y_2 \in \mathcal Y$, the representations $\lambda_{G/Y_1}, \lambda_{G/Y_2}$ are weakly equivalent. In particular $\mathcal Y$ yields a canonical quotient of the maximal $C^*$-algebra of $G$, defined by 
$$C^*(\mathcal Y) = C^*_{\lambda_{G/Y}}(G),$$ 
where $Y$ is an arbitrary element of $\mathcal Y$. This fact has recently been observed in the case of a discrete group $G$ by T.~Kawabe~\cite{Kawabe} and G.~Elek~\cite{Elek}. 

\begin{rmk}\label{rem:KaKu}
The fact that the set $X_0$ of continuity points of the stabilizer map $X \to \Sub(G)$ is a dense $G_\delta$ holds more generally, for the same reason, if there is a closed metrizable subset $\mathcal M \subseteq \Sub(G)$ such that $G_x \in \mathcal M$ for all $x \in X$. This happens notably in the case where the quotient of $G$ by the kernel $W$ of the $G$-action on $X$ is  second countable: indeed the closed subset $\{H  \in \Sub(G) \mid H \geq W\} \subseteq \Sub(G)$ is homeomorphic to $\Sub(G/W)$, and is thus metrizable. 
\end{rmk}

\section{Koopman and quasi-regular representations}
This section is devoted to the proofs of Theorems~\ref{thm:weak-containment} and~\ref{thm:Ananth} stated in the Introduction. We shall occasionaly use the shorthand $\pi \prec\sigma$ to denote that  $\pi$ is weakly contained in $\sigma$.

\subsection{The Koopman representation weakly contains a quasi-regular representation}

In the setting of Theorem~\ref{thm:weak-containment}, we invoke the existence of a \textbf{compact model}, i.e. a metrizable compact  space $X'$ with a continuous $G$-action, and a quasi-invariant probability measure $\nu'$ on $X'$, such that the $G$-spaces $(X, \nu)$ and $(X', \nu')$ are measurably isomorphic. This well-known fact follows from~\cite[Theorem~3.2]{Varad}, which affords the metrizable compact $G$-space $X'$ and a  $G$-equivariant isomorphism of Borel spaces  $\phi \colon X \to X'_0$, where $X'_0$ is a $G$-invariant Borel subset. The measure $\nu'$ can then be defined by setting $\nu' = \phi_*(\nu)$. Given the existence of compact models, Theorem~\ref{thm:weak-containment} stated in the introduction directly follows from the topological version of that result. 

\begin{thm}\label{thm:weak-cont-topol}
Let $G$ be a second countable tdlc group, $X$ be a locally compact $G$-space, and $\nu$ a $G$-quasi-invariant $\sigma$-finite Radon measure on $X$ such that $L^2(X, \nu)$ is separable. Let $\kappa$ be the Koopman representation of $G$ on $L^2(X, \nu)$. 

If $x \in X$ is a point belonging to the support of $\nu$ and such that the stabilizer $G_x$ is regionally elliptic, then the quasi-regular representation $\lambda_{G/G_x}$ is weakly contained in the Koopman representation $\kappa$.
\end{thm}

\begin{proof}
We choose a decreasing sequence $U_n$ of relatively compact open subsets of $G$ forming a neighbourhood basis at $e$. We further choose an increasing sequence of compact subsets $K_n$ covering $G$. Finally, let $Q_n$ be an increasing sequence of compact subgroups of $G_x$ covering $G_x$. Note that $K_n$ and $Q_n$ ultimately cover any compact subsets of $G$ and $G_x$, respectively.
Define $K'_n = K_n \setminus U_n G_x$; this is a compact set. Therefore, by Lemma~\ref{lem:stabilizer}, there is a neighbourhood $\alpha_n$ of $x$ such that $K'_n \alpha_n \cap \alpha_n = \varnothing$.

For each $n$, we now choose an open neighbourhood $\beta_n$ of $x$ such that $g^{-1}\beta_n \subseteq \alpha_n$ holds for all $g\in Q_n$. This is possible since $Q_n$ is compact and $G_x$ fixes $x$. The characteristic function $\one_{\beta_n}$ is a non-negative element of $L^1(X, \nu)$ which is non-zero since $x$ is in the support of $\nu$. Now the element
$$\xi_n = \kappa_\nu(\one_{Q_n}) \one_\beta = \int_{Q_n}  \kappa_\nu(g) \one_\beta \, dg \, \in \, L^1(X, \nu)$$
is non-negative, non-zero, supported in $\alpha$, and $Q_n$-invariant.

We now define measures $\mu'_n$ and $\mu_n$ on $G$ as follows. Let $d\mu'_n(g) = (\kappa(g)\xi_n, \xi_n) dg$ for some choice of right Haar measure on $G$ and $\mu_n = \mu'_n / \mu'_n(U_0)$. Note that $\mu_n$ is invariant under right translation by $Q_n$, and vanishes on $K'_n$ by the choice of $\xi_n$.

We first claim that for any compact set $\Omega\subset G$, there is $m$ such that for all $n\geq m$ we have
$$\Omega \subseteq K'_n \cup  U_n Q_m.$$
Indeed, take $m$ large enough to have $\Omega \subseteq K_m$ and $G_x \cap U_0^{-1} \Omega \subseteq Q_m$. Then any $\omega\in\Omega \setminus K'_n$ can be written $\omega \in u h$ for some $u\in U_n$ and $h\in G_x$. This $h$  is in $U_n^{-1}\Omega$ and hence in $Q_m$, so that $\omega\in  U_n Q_m$ as claimed.

Next, we claim that the sequence $\mu_n(\Omega)$ remains bounded for any given compact set $\Omega\subset G$.
Fix $m$ as in the first claim and consider any $n\geq m$. Since the sequence $U_m$ is decreasing and a basis of relatively compact neighbourhoods, we can assume $\overline{U_m} \subseteq U_0$ after possibly increasing $m$. Since the first claim implies in particular $\Omega \subseteq K'_n \cup  U_m Q_m$, it suffices to bound $\mu_n(U_m  Q_m)$ independently of $n$. By compactness of $\overline{U_m} Q_m$, there are $q_1, \ldots q_r\in Q_m$ (independent of $n$) such that $U_m  Q_m$ is in the union of the $r$ translates $U_0 q_i$. Since each $\mu_n$ is (right) $Q_m$-invariant, it follows $\mu_n(U_m  Q_m)\leq r\mu_n(U_0)$ for all $n$, and the claim follows.

Upon passing to a subsequence, it follows from Lemma~\ref{lem:VagueConv} that $(\mu_n)$ vaguely converges to a positive definite measure $\mu$. Since $\mu_n(U_0)=1$ for all $n$, we have $\mu(U) = 1$. Since $\mu_n(\Omega \cap K'_n)= 0$ for all $n$, it follows that $\mu$ is supported on $G_x$. Moreover, we observe that $\mu$ is (right) $G_x$-invariant since $G_x = \bigcup_n Q_n$ and since $\mu_n$ is (right) $Q_n$-invariant for all $n$. Hence, by~\cite[Chapter~VIII, \S5.1, Lemma~1]{Bourbaki}, we infer that $\mu$ is a Haar measure on $G_x$. Note   that $G_x$ is unimodular since it is regionally elliptic. Moreover, since every element of $G_x$ is contained in a compact subgroup, we see that $G_x$ is contained in $\Ker(\Delta_G)$, where $\Delta_G$ is the modular function of $G$. Therefore, we have $\sqrt{\Delta_G} d\mu = d\mu$, so that  the representation $\pi_\mu$ is equivalent to the quasi-regular representation $\lambda_{G/G_x}$ by   Theorem~\ref{thm:Blattner}.

Finally, we notice that the unitary representation $\pi_{\mu_n}$ is contained in $\kappa$ for all $n$, since $\pi_{\mu_n}$ is the representation associated with the positive definite function $g \mapsto (\kappa(g)\xi_n, \xi_n)$ via the GNS construction. From Lemma~\ref{lem:weak-limit-posdef}, we deduce $\pi_\mu \prec \kappa$, which completes the proof.
\end{proof}
	

\subsection{The Koopman representation is weakly contained in a quasi-regular representation}

In this section we prove Theorem~\ref{thm:Ananth}. We will need the following general fact which should be of independent interest.

\begin{prop}\label{prop:amen-stab-quas}
Let $G$ be a second countable locally compact group, and $X$ a minimal locally compact $G$-space. Assume $x\in X$ is such that the stabilizer $G_x$ is amenable. Then $\lambda_{G/G_x} \prec \lambda_{G/G_y}$ for every $y\in X$.
\end{prop}
\begin{proof}
Let $y\in X$ be arbitrary.
Since the $G$-action on $X$ is minimal, there is a sequence $(g_n)$ in $G$ such that $(g_ny)$ converges to $x$. By passing to a subsequence, we may assume that $G_{g_ny}$ Chabauty converges to some subgroup $H < G_x$. By Fell's continuity theorem (\Cref{thm:Fell}), the quasi-regular representation $\lambda_{G/G_{g_ny}}$ converges to $\lambda_{G/H}$. Since $\lambda_{G/G_{g_ny}} = \lambda_{G/g_nG_{y}g_n^{-1}}$ is equivalent to $\lambda_{G/G_y}$ for all $n$, we deduce $\lambda_{G/H} \prec\lambda_{G/G_y}$.

Since $G_x$ is amenable, the trivial representation $\one_{G_x}$ of $G_x$  is weakly contained in $\lambda_{G_x/H}$ (see \cite[${\rm n}^o~2, \S~4$]{Eym}). Inducing up to $G$, we deduce $\lambda_{G/G_x} \prec \lambda_{G/H}$. We have seen that the latter is weakly contained in $\lambda_{G/G_y}$. Thus the proof is complete.
\end{proof}

\begin{proof}[Proof of Theorem~\ref{thm:Ananth}]
Denote by $X_1 \subseteq X$ the conull subset consisting of those points $x\in X$ whose stabilizer $G_x$ is amenable. Since the $G$-action on $X$ is minimal by hypothesis, the assertion (i) follows from Proposition~\ref{prop:amen-stab-quas}.

We shall now prove the assertion (ii). 

Fix a regular probability measure $\mu$ on $G$ in the class of the Haar measures. Then for every $x\in X$, the measure $\mu_x:=p^x_*(\mu)$ is quasi-invariant, where $p^x: G\to G/G_x$ is the canonical projection.

We endow the product $G \times X$ with a $G$-action defined by $g\!\cdot\!(h, x) = (hg^{-1}, gx)$. This action preserves the class of the measure $\mu \times \nu$. 

Let now $Y\subset X\times X$ be the orbit equivalence relation for the $G$-action on $X$. We endow $Y$ with the $G$-action defined by $g\!\cdot\!(x, y) = (gx, y)$. 
Consider the Borel map $p \colon G \times X \to Y$ defined by $p(g,x) = (x, gx)$. Observe that $p$ is $G$-equivariant. We define $\nu_Y = p_*(\mu\times\nu)$, so that $\nu_Y$ is quasi-invariant under $G$.
Moreover, we have $\nu_Y = \int_X (\delta_x \times \mu_x)\,d\nu(x)$, where we have identified the $G$-orbit of $x$ with $G/G_x$. This implies that the Koopman representation $\kappa_{\nu_Y}$ is equivalent to $\int_X \lambda_{G/G_x} \,d\nu(x)$. 

Since the action of $G$ on $(X, \nu)$ is amenable in the sense of Zimmer, we have $\nu(X_1)= 1$. In view of the assertion (i), it follows that $\int_X \lambda_{G/G_x} \,d\nu(x)$ is weakly equivalent to $\lambda_{G/G_z}$ for each $z \in X_1$. 
Therefore $\kappa_{\nu_Y}$ is weakly equivalent to $\lambda_{G/G_z}$ for any $z\in X_1$.

Now, consider the projection $q \colon Y\to X$ onto the first coordinate. The map $q$ is Borel measurable and $G$-equivariant. Moreover, it follows from the definitions that  $q_*(\nu_Y) = \mu*\nu$. By Lemma~\ref{lem:conv->equiv-meas} we have $\mu*\nu \sim\nu$. Since the action of $G$ on $(X, \mu*\nu)$ is amenable in the sense of Zimmer (see~\cite[Proposition~3.3.5]{ADR}, and~\cite[Theorem~6.8]{Adams} or~\cite{Kaimanovich}), it follows from~\cite[Proposition~4.3.2 \& Theorem~3.2.1]{AD03} that $\kappa_{\nu_Y}$ weakly contains the Koopman representation of $G$ on $L^2(X, \mu*\nu)$. Therefore, using Lemma~\ref{lem:equiv-meas->equiv-koop},  we conclude $\kappa_{\nu}\prec \lambda_{G/G_z}$ for all $z\in X_1$. 
%
%
\end{proof}


\section{Hyperbolic locally compact groups}

A locally compact group $G$ is called \textbf{hyperbolic} if it has a compact generating set with respect to which the word metric on $G$ is Gromov-hyperbolic. By~\cite[Proposition~2.1]{CCMT}, this is equivalent to requiring that $G$ has a continuous, proper, cocompact, isometric action on a locally compact geodesic metric space $X$ that is Gromov-hyperbolic. This space is automatically proper by the Hopf--Rinow theorem, since cocompactness ensures completeness.

In this section, we assume that $G$ is locally compact hyperbolic and fix a space $X$ as above.

We freely refer to Gromov's typology concerning isometries  and isometric group actions on hyperbolic spaces, as recalled in e.g.~\cite[\S3.1]{CCMT}. We recall that the \textbf{displacement length} of an isometry $g$ is $|g|=\inf\{ d(g x, x) : x\in X\}$ and its \textbf{asymptotic displacement length} is $|g|_\infty=\lim_{n\to \infty} \frac1n d(g^n x, x)$, which does not depend on $x$. We have $|g|_\infty \leq |g|$ and $|g| \leq |g|_\infty + 16\delta$, where $\delta$ is a hyperbolicity constant for $X$. Moreover, $|g|_\infty$ is positive if and only if $g$ is hyperbolic. For all this, see~\cite[Chapter~10 \S6]{CDP}. 

We also use throughout that given $x\in X$ and $\xi\neq \xi'\in \partial X$, there exists a geodesic ray from $x$ to $\xi$ and a geodesic line from $\xi'$ to $\xi$. See e.g.\ Proposition~4 in~\cite[Chap.~7]{Ghys-Harpe} or~\cite[III.H.3.1]{BriHae} for the former and Proposition~6 in~\cite[Chap.~7]{Ghys-Harpe} or~\cite[III.H.3.2]{BriHae} for the latter.

Whenever a constant can be chosen depending only on $X$, we shall call it an \textbf{$X$-constant}; most of the time it even depends only on the hyperbolicity constant of $X$.

\subsection{The structure of relatively amenable subgroups}\label{sec:hyp-rel-amen}

Our goal is to describe the algebraic structure of (relatively) amenable closed subgroups of the hyperbolic locally compact group $G$. Similar results for groups acting properly cocompact on proper CAT(0) spaces can be found in~\cite{CM-fixed}. We first describe the geometric features of the action of a relatively amenable subgroup, and then derive algebraic information. 

\begin{prop}\label{prop:RelAmen-action}
For any closed subgroup $H \leq G$, the following assertions are equivalent. 
\begin{enumerate}[label=(\roman*)]
\item $H$ is amenable. 
\item $H$ is relatively amenable in $G$. 
\item $H$ is compact, or $H$ fixes a point $\xi \in \partial X$, or $H$ stabilizes a pair $\{\xi, \xi'\} \subseteq \partial X$. 
\end{enumerate}
\end{prop}
\begin{proof}
That (i) implies (ii) is true in full generality. 

Assume that $H$ is relatively amenable. If the $H$-action  on $X$ is bounded, then $H$ is compact since the $G$-action is proper and $H$ is closed by hypothesis.  If the $H$-action is horocyclic or focal (resp. lineal), then  $H$ fixes a point $\xi \in \partial X$ (resp. $H$ stabilizes a pair $\{\xi, \xi'\} \subseteq \partial X$). Therefore, in order to show that (iii) holds, it suffices to prove that the $H$-action on $X$ cannot be of general type. Suppose for a contradiction that it is. Then $H$   contains a Schottky subgroup $\Lambda$ (see~\cite[Lemma~3.3]{CCMT}). The action of a Schottky subgroup $\Lambda$ on $\partial X$ is minimal and strongly proximal, so that $\Lambda$, and hence also $H$, does not fix any probability measure on $\partial X$. This implies that $H$ is not relatively amenable in $G$, thereby confirming that (ii) implies (iii). 

That (iii) implies (i) is a consequence of a well-known result of S.~Adams, cited as Lemma~3.10 in~\cite{CCMT}, ensuring that the stabilizer of each point $\xi \in \partial X$ in the full isometry group of $X$ is amenable. 
\end{proof}

Our next goal is to describe the algebraic structure of the stabilizer $G_\xi$ of a boundary point $\xi$, or the stabilizers  $G_{\{\xi, \xi'\}}$ of a boundary pair. To that end, we shall use the \textbf{Busemann character} $\beta_\xi \colon G_\xi \to \mathbf R$ afforded by~\cite[Corollary~3.9]{CCMT}. It is a continuous homomorphism whose kernel is denoted by $G_\xi^0$. Moreover, by~\cite[Lemma~3.8]{CCMT}, the subgroup  $G_\xi^0$ consists of those elements $g \in G_\xi$ that act as elliptic or parabolic isometries on $X$; in other words an element $g \in G_\xi$ satisfies $\beta_\xi(g) \neq 0$ if and only if it is a hyperbolic isometry of $X$.

\begin{lem}\label{lem:2-ended}
Let $\xi \neq \xi' \in \partial X$. Then $G^0_{\xi, \xi'} := G^0_\xi \cap G_{\xi'} = G_\xi \cap G^0_{\xi'}$ is a compact normal subgroup of $G_{\{\xi, \xi'\}}$. Moreover, the quotient group $G_{\{\xi, \xi'\}}/G^0_{\xi, \xi'}$ is trivial or isomorphic to $\mathbf Z$, $\mathbf Z \rtimes \{\pm 1\}$, $\mathbf R$ or $\mathbf R \rtimes \{\pm 1\}$.
\end{lem}
\begin{proof}
Notice that $G^0_\xi \cap G_{\xi'}$ consists of the elements of $G_\xi \cap G_{\xi'}$ that are not hyperbolic. The same description applies to $G_\xi \cap G^0_{\xi'}$. Hence $G^0_\xi \cap G_{\xi'} = G_\xi \cap G^0_{\xi'}$. Denoting that subgroup by $W$, we may complete the proof  by invoking  the same argument as in the proof of~\cite[Proposition~5.6]{CCMT}. 
\end{proof}

Following~\cite[\S4.1]{CM-fixed}, we say that a subgroup $A \leq G$ is \textbf{compactible} with limit $K$ if $K \leq G$ is a compact subgroup such that for any neighbourhood $U$ of $K$ in $G$ and every finite subset $F \subseteq A$, there is $\alpha \in \Aut(G)$ with $\alpha(F) \subseteq U$. If this is the case, we say that a sequence $(g_n)$ in $G$ is  a \textbf{compacting sequence} for $A$ if  for all $a \in A$, the sequence $(g_n a g_n^{-1})$ is bounded and each of its accumulation point belongs to $K$. 

\begin{lem}\label{lem:compacting-seq}
Let $\xi \in \partial X$. There exists a sequence $(g_n)$ in $G$ and a pair $\eta \neq \eta'$ in $\partial X$ such that $G^0_\xi$ is compactible with limit $G^0_{\eta, \eta'}$, and $(g_n)$ is a compacting sequence for $G^0_\xi$. 
\end{lem}
\begin{proof}
Fix a geodesic ray $r \colon \mathbf R_+ \to X$ with endpoint $\xi$. Let $g_n \in G$ be such that the sequence $g_n(r(n))$ is bounded. Upon extracting, we may assume that $g_n r(\psi(n) +t)$ converges for every $t\in\Rb$, where $\psi \colon \mathbf N \to \mathbf N$ is a strictly increasing function. More precisely, note that for any fixed $t$, the point $r(\psi(n) +t)$ is defined when $n$ is large enough and $g_n r(\psi(n) +t)$ remains in a bounded set. Finding a common subsequence for all $t\in\Rb$ can be achieved e.g.\ by applying a diagonal argument for $t\in\Qb$ and then using that $r$ is uniformly continuous. We can further assume that $g_n(\xi)$ converges to a point $\eta \in \partial X$ and that $g_n(r(0))$ converges to $\eta' \in \partial X$. There is now a geodesic line $\ell$ joining $\eta'$ to $\eta$, namely $\ell(t) = \lim_n g_n r(\psi(n) +t)$. In particular we have $\eta \neq \eta'$. It now follows from the definition that for each $a \in G^0_\xi$, the sequence $(g_n a g_n^{-1})$ is bounded (because the sequence $n \mapsto d(a g_n^{-1}(\ell(0)), g_n^{-1}(\ell(0)))$ is bounded), and that each of its accumulation points belongs to $G_{\eta, \eta'}$. 

We claim that each accumulation point of $(g_n a g_n^{-1})$ actually belongs to $G^0_{\eta, \eta'}$. To prove that claim, we shall use the existence of an $X$-constant $K$ such that for each $g \in G^0_\xi$, we have $d(gr(n), r(n)) \leq K$ for all sufficiently large $n$ (see~\cite[Lemma~21]{BMW}). 
Let now $t \in G_{\eta, \eta'}$ and a subsequence $(g_{\phi(n)} a g_{\phi(n)}^{-1})$ converging to $t$ and such that $\big(g_{\phi(n)}(r(n))\big)$ converges to some point $x \in X$. For each integer $N$, the sequence  $(g_{\phi(n)} a^N g_{\phi(n)}^{-1})$ converges to $t^N$. We have 
\begin{align*}
d(t^N(x), x) & = \lim_n d(g_{\phi(n)} a^N g_{\phi(n)}^{-1}(x), x)\\
& = \lim_n  d(a^N g_{\phi(n)}^{-1}(x), g_{\phi(n)}^{-1}(x))\\
& = \lim_n d(a^N r(n), r(n))\\
& \leq K.
\end{align*}
This ensures that the isometry $t$ is elliptic, thereby proving the claim. 

Since   $G^0_{\eta, \eta'}$ is compact by Lemma~\ref{lem:2-ended}, we conclude that $(g_n)$ is indeed a compacting sequence for $G^0_\xi$, and that $G^0_\xi$ is compactible with limit  $G^0_{\eta, \eta'}$. 
\end{proof}

\begin{prop}\label{prop:tdlc-reg-ell}
Assume that $G$ is totally disconnected. Then, for each $\xi \in \partial X$, the group $G_\xi^0$ is regionally elliptic. 
\end{prop}
\begin{proof}
In view of Lemma~\ref{lem:compacting-seq}, this is a direct consequence of~\cite[Proposition~4.2]{CM-fixed}.
\end{proof}

We close this subsection by recording the following result from~\cite{CCMT}. It will be used repeatedly to various questions on hyperbolic locally compact groups  to the totally disconnected case.

\begin{prop}\label{prop:non-amen-hyper}
Let $G$ be a non-amenable hyperbolic locally compact group. Then $G$ has a largest compact normal subgroup $W$ and the quotient $G/W$ is either a virtually connected rank one simple Lie group, or a totally disconnected group. 
\end{prop}
\begin{proof}
Since a non-amenable hyperbolic locally compact group is of general type by Proposition~\ref{prop:RelAmen-action}, the result follows from~\cite[Proposition~5.10]{CCMT}. 
\end{proof}

\begin{rmk}
In a $\sigma$-compact locally compact group $G$, every identity neighbourhood contains a compact normal subgroup $N$ such that $G/N$ is second countable, see~\cite{KaKu}. In particular the quotient $G/W$ appearing in Proposition~\ref{prop:non-amen-hyper} is second countable. Since $W$ acts trivially on the Gromov boundary of $G$, it follows from Remark~\ref{rem:KaKu} that the set of continuity points of the stabilizer map $\partial G \to \Sub(G)$ is a dense $G_\delta$, although the group $G$ itself need not be second countable. We shall frequently use this fact without further notice. 
\end{rmk}

\subsection{Ballistic boundary points}
The structure of hyperbolic locally compact groups with a cocompact amenable subgroup has been described in~\cite{CCMT}. If the group is elementary, it is either compact, or $2$-ended, in which case it is described by~\cite[Proposition~5.6]{CCMT}. If the group is non-elementary, it is described by Theorem~A or D in~\cite{CCMT}, depending on whether the group is amenable or not. For a non-amenable locally compact hyperbolic group $G$, the existence of a cocompact amenable subgroup is equivalent to the fact that $G$ acts $2$-transitively on its Gromov boundary (see~\cite[Theorem~8.1]{CCMT}). 

The following theorem establishes an additional powerful criterion, namely it suffices to verify $G_\xi^0 \neq G_\xi$ for suitable points at infinity $\xi$. We shall call $\xi$ a \textbf{ballistic} point when $G_\xi^0 \neq G_\xi$ is satisfied.

\begin{thm}\label{thm:Ballistic}
Let $G$ be a non-amenable hyperbolic locally compact group and $X$ be a proper geodesic metric space on which $G$ acts continuously, properly and cocompactly by isometries.

The following conditions are equivalent.

\begin{enumerate}[label=(\roman*)]
\item\label{it:coco} $G$ has a cocompact amenable subgroup. 

\item\label{it:bd-2-tr} The $G$-action on the Gromov boundary $\partial X$ is $2$-transitive.

\item\label{it:all-ballistic} For all $\xi \in \partial X$, we have $G_\xi^0 \neq G_\xi$. 

\item\label{it:generic-ballistic} For some continuity point $\eta \in \partial X$ of the stabilizer map, we have $G_\eta^0 \neq G_\eta$. 

\item\label{it:bdd-tsl} There is a constant $K$ such that, for every hyperbolic element $\gamma \in G$, there exists a hyperbolic element $\gamma' \in G$ with asymptotic displacement length~$|\gamma'|_\infty \leq K$ such that $\gamma$ and $\gamma'$ share the same pair of fixed points in $\partial X$.

\end{enumerate}
\end{thm}

As indicated above, the equivalence between~\ref{it:coco} and~\ref{it:bd-2-tr} is taken from~\cite{CCMT}. A special instance of the equivalence between~\ref{it:bd-2-tr} and~\ref{it:all-ballistic} has been observed for specific families of groups of tree automorphisms by C.~Ciobotaru in her PhD thesis~\cite[Proposition~2.2.11]{Cio}. 

\bigskip

Before embarking on the proof, we record the following classical criterion to identify hyperbolic isometries.

\begin{lem}\label{lem:hyp-isom}
Let $g$ be an isometry of a $\delta$-hyperbolic metric space and fix any $a>0$.

If there exists a point $x$ with $d(g^{-1} x, g x) \geq d(x, g x) + 2\delta +a$, then $g$ is hyperbolic with $|g|_\infty\geq a$.
\end{lem}

\begin{proof}
The fact that $g$ is hyperbolic is proved in~\cite[Chapter~9, Lemma~2.2]{CDP}. The proof given there proceeds by establishing $d(g^n x, x) \geq n a$ for all positive integers $n$, which yields also the estimate on $|g|_\infty$.
\end{proof}

There is a sort of converse upon replacing $g$ by a power:

\begin{lem}\label{lem:hyp-converse}
If $g$ is hyperbolic and $a>0$, then for every $x$ there is $n\geq 1$ with $d(g^{-n} x, g^n x) \geq d(x, g^n x) + 2\delta +a$.
\end{lem}
\begin{proof}
Fix $g$, $a$ and $x$ and suppose there is no such $n$. Then, plugging the inequality
$$d(x, g^{2n} x) = d(g^{-n} x, g^n x) < d(x, g^n x) + 2\delta +a$$
in the definition of $|g|_\infty$ shows $|g|_\infty\leq \frac12 |g|_\infty$, hence $|g|_\infty =0$ and $g$ is not hyperbolic.
\end{proof}

We also have the following behaviour with respect to limits in $\Isom(X)$.

\begin{lem}\label{lem:hyp-limit}
The set of hyperbolic elements is open in $\Isom(X)$. 

Moreover, given a sequence $(g_n)$ converging to  $g\in \Isom(X)$, we have  $\limsup_n |g_n|\leq |g|$ and $\limsup_n |g_n|_\infty \leq |g|_\infty + 16 \delta$.
\end{lem}

\begin{proof}
Suppose first that $(g_n)$ converges to a hyperbolic isometry $g$. By Lemma~\ref{lem:hyp-converse}, we can choose $x$ and $n_0$ such that $d(g^{-n_0} (x), g^{n_0}(x)) \geq d(x, g^{n_0}( x)) + 2\delta +2$. It follows that for all $n$ large enough we have $d(g_n^{-n_0}( x), g_n^{n_0}(x)) \geq d(x, g_n^{n_0} (x)) + 2\delta +1$. Now Lemma~\ref{lem:hyp-isom} implies that $g_n^{n_0}$ is hyperbolic, and hence so is $g_n$.

Regardless of the type of isometries, the first estimate follows from the definition of the displacement length and the second from the estimates between $|\cdot|$ and $|\cdot|_\infty$ recalled earlier.
\end{proof}

Within the stabilizer of a point at infinity, we have more continuity because the asymptotic displacement length can be read on the (continuous) Busemann character as follows.

\begin{prop}\label{prop:length-Busemann}
Let $g$ be an isometry fixing $\xi\in\partial X$. Then $|g|_\infty = |\beta_\xi(g) |$.
\end{prop}

Just like the three above lemmas, this proposition holds for any geodesic hyperbolic space $X$; however, we should recall that if one applies it beyond the \emph{proper} case (to which locally compact groups reduce us here), then $\beta_\xi$ is a priori only a continuous homogeneous quasimorphism rather than a homomorphism (cf.~\cite[\S3]{CCMT}). This does not affect the proposition nor its proof, which relies on the following.

\begin{lem}\label{lem:orbit-Busemann}
Let $g$ be an isometry fixing $\xi\in\partial X$ and let $x\in X$. If $g$ is not parabolic, then the difference $n |\beta_\xi(g)| - d(g^n x, x)$ remains bounded over $n \geq 0$.
\end{lem}
\begin{proof}[Proof of Lemma~\ref{lem:orbit-Busemann}]
Since $\beta_\xi(g)$ vanishes if and only if $g$ is non-hyperbolic~\cite[Lemma~3.8]{CCMT}, we can assume $g$ hyperbolic. Upon replacing $g$ with its inverse, we can further assume that $\xi$ is the attracting point of $g$, or equivalently $\beta_\xi(g) >0$ (see again~\cite[Lemma~3.8]{CCMT}). Upon enlarging the desired bound, we suffer no loss of generality when replacing $g$ by a (fixed) power of itself before varying $n$. Therefore, using Lemma~\ref{lem:hyp-converse}, we can assume $d(g^2 x, x) \geq d(g x,x) + 2\delta+1$. Then the orbit $(g^n x)$ is quasigeodesic, see e.g.\ the proof of Lemma~2.2 in~\cite[Chap.~9]{CDP}. By stability of quasigeodesics (Theorem~25(i) in~\cite[Chap.~5]{Ghys-Harpe}), it follows that there is a constant $D$ such that $(g^n x)_{n\geq 0}$ remains at distance at most $D$ from a geodesic ray $r$. Note that $r(+\infty)=\xi$ since $g^n x\to \xi$, and upon increasing $D$ we can assume $r(0)=x$. We choose $s_n\geq 0$ such that $d(r(s_n), g^n x) \leq D$ for all $n$.

According to Proposition~3.7 in~\cite{CCMT}, we can write $\beta_\xi(g) = \lim_n \frac1n h(x, g^n x)$ where the function $h(x,y)$ of $x,y$ is any accumulation point (for the pointwise convergence) of $d(x, r(s)) - d(y, r(s))$ as $s\to\infty$. Note that $d(x, r(s)) = s$, while $\big|d(g^n x, r(s)) - |s-s_n|\big|$ is bounded by $D$. Letting $s\to \infty$, we deduce that  $|h(x, g^n x) - s_n|$ is bounded by $D$. Using one more time $d(r(s_n), g^n x) \leq D$, we conclude
$$\Big| h(x, g^n x)  - d(x, g^n x) \Big| \leq 2 D.$$
It is also shown in Proposition~3.7 of~\cite{CCMT} that the difference $\big|h(x, g^n x) - \beta_\xi(g^n)\big|$ is bounded independently of $n$; now the lemma follows since $\beta_\xi(g^n) = n \beta_\xi(g)$.
\end{proof}

\begin{proof}[Proof of Proposition~\ref{prop:length-Busemann}]
We can assume that $g$ is hyperbolic thanks to~\cite[Lemma~3.8]{CCMT}. The proposition then follows from Lemma~\ref{lem:orbit-Busemann} by dividing by $n$ and letting $n\to\infty$.
\end{proof}

As mentionned in the proof of Lemma~\ref{lem:orbit-Busemann}, it is well-known that the orbits of hyperbolic isometries are quasigeodesic. We can now stregthen this conclusion as follows, removing any multiplicative error factor.

\begin{cor}\label{cor:Z-action-homothety}
Let $g$ by an isometry of a geodesic hyperbolic space $X$ and let $x\in X$.  If $g$ is not parabolic, then the difference
$$ |g|_\infty \cdot |n-m|  - d(g^n x, g^mx) $$
remains bounded over $n, m\in \Zb$.
\end{cor}
\begin{proof}
Since $g$ is isometric, it suffices to show the statement for $m=0$ and $n\geq 0$. We can again assume that $g$ is hyperbolic; in particular it fixes some $\xi\in\partial X$. Now the statement follows by combining Lemma~\ref{lem:orbit-Busemann} with Proposition~\ref{prop:length-Busemann}.
\end{proof}

\begin{lem}\label{lem:hyp-uniform}
Let $G$ and $X$ be as in Theorem~\ref{thm:Ballistic} and assume that~\ref{it:generic-ballistic} holds. 
There is a constant $K$ such that for each $\xi \in \partial X$, there exists a hyperbolic isometry $t \in G_\xi$ with $|t|_\infty \leq K$. 
\end{lem}
\begin{proof}
Since $G_\eta^0 \neq G_\eta$, there is $\gamma \in G_\eta$ hyperbolic. We can take any $K > |\gamma|_\infty + 16 \delta$. Let indeed $\xi \in \partial X$ be arbitrary. By hypothesis, the group $G$ is non-amenable, hence the $G$-action on $X$ is of general type. This implies that  the $G$-action on $\partial X$ is minimal. Therefore, there exists $(g_n)$ in $G$ such that the sequence  $(g_n \xi)$ converges to $\eta$. By the choice of $\eta$, we know that the sequence $(g_n G_\xi g_n^{-1})$ Chabauty converges to $G_\eta$. Therefore there exist $t_n \in G_\xi$ such that the sequence $(g_n t_n g_n^{-1})$ converges to $\gamma$. For all $n$ large enough, it follows from Lemma~\ref{lem:hyp-limit} that $g_n t_n g_n^{-1}$ is hyperbolic, with asymptotic displacement length at most $K$. Since the asymptotic displacement length is invariant under conjugation, we deduce that $t_n$ is hyperbolic with $|t_n|_\infty \leq K$.
\end{proof}

We shall use the notion of \textbf{duality} due to Chen and Eberlein (compare~\cite[III.1]{Ballmann}). We say that a pair $(\xi, \xi')$ in $\partial X \times \partial X$ is \textbf{$G$-dual} if there exist a sequence $(\gamma_n)$ in $G$ such that $\gamma_n(x_0) \to \xi$ and $\gamma_n^{-1}(x_0) \to \xi'$ for some (hence any) $x_0 \in X$. Since $G$ is supposed to be non-amenable, its action on $X$ is of general type by Proposition~\ref{prop:RelAmen-action}. 

\begin{lem}\label{lem:dual}
If the $G$-action on $X$ is of general type, then every pair $(\xi, \xi') \in \partial X \times \partial X$ is $G$-dual. 
\end{lem}

\begin{proof}
There exists at least \emph{some} $G$-dual pair $(\xi, \xi')$, for instance the attracting and repelling points of some hyperbolic isometry. For this given $\xi' \in \partial X$, the non-empty collection of those $\xi \in \partial X$ such that $(\xi, \xi')$ is $G$-dual is closed and $G$-invariant. Since the $G$-action on $\partial X$ is minimal, we deduce that the pair $(\xi, \xi')$ is $G$-dual for all $\xi \in \partial X$. By symmetry, it follows that every pair $(\xi, \xi')$ is $G$-dual. 
\end{proof}

We now establish two strengthenings of Lemma~\ref{lem:hyp-uniform}.

\begin{lem}\label{lem:hyp-bi-uniform}
Let $G$ and $X$ be as in Theorem~\ref{thm:Ballistic}. If the assertion~\ref{it:generic-ballistic} from the latter statement holds, then~\ref{it:bdd-tsl} holds as well. 
\end{lem}

\begin{proof}
Since $G_\eta^0 \neq G_\eta$, there is $\gamma \in G_\eta$ hyperbolic. 
Let $\xi \in \partial X$ be the repelling fixed point of $\gamma$. By Lemma~\ref{lem:hyp-uniform}, the group $G_\xi$ contains a hyperbolic element $t$ with $|t|_\infty\leq K$. Notice that the sequence $(\gamma^n t \gamma^{-n})$ is bounded in $G$. Upon extracting, it converges to an element $\gamma' \in G_\xi$. Since the Busemann character $\beta_\xi \colon G_\xi \to \mathbf R$ is continuous and  $\beta_\xi(\gamma^n t \gamma^{-n}) = \beta_\xi(t)$, we have $\beta_\xi(\gamma') = \beta_\xi(t)$. Thus Proposition~\ref{prop:length-Busemann} implies $|\gamma'|_\infty = |t|_\infty$, which shows both that $\gamma'$ is hyperbolic (since this length is positive) and that $|\gamma'|_\infty \leq K$.

Finally, if $\xi'$ denotes the attracting fixed point of $\gamma$ and $\eta$ the fixed point of $t$ different from $\xi$, we have $\lim_n\gamma^n\eta = \xi'$, so that $\gamma'$ fixes $\xi'$. 
\end{proof}

\begin{prop}\label{prop:hyp-two-uniform}
Let $G$ and $X$ be as in Theorem~\ref{thm:Ballistic} and suppose that  the assertion~\ref{it:bdd-tsl}  from the latter statement holds. 

There is a constant $L$ such that, for any pair $(\xi, \xi')$ in $\partial X \times \partial X$ with $\xi \neq \xi'$, there is a hyperbolic isometry $\gamma \in G$ with $|\gamma|_\infty\leq L$ fixing $\xi$ and $\xi'$. 
\end{prop}

The following general lemma will be needed in the proof and again later.

\begin{lem}\label{lem:proj-on-geod}
Let $X$ be a geodesic $\delta$-hyperbolic space; there exist constants $H, J$ depending only on $\delta$ with the following property.

Let $\ell\subseteq X$ be a geodesic line (the image of a geodesic $\Zb\to X$) and choose a map $P\colon X\to \ell$ such that $d(P(x), x)$ minimizes the distance from $x$ to $\ell$ for any $x\in X$.

If $x,y\in X$ satisfy $d(P(x), P(y)) \geq J$, then any choices of geodesic segments
$$[x, P(x)] \cup [P(x), P(y)]\cup [P(y), y]$$
remain within distance less than $H$ from any geodesic segment $[x,y]$.
\end{lem}

\begin{proof}[Proof of Lemma~\ref{lem:proj-on-geod}]
We recall that a $(\lambda, C, J)$-\textbf{local quasigeodesic} segment is a path satisfying the $(\lambda, C)$-quasigeodesic conditions for parameters less than $J$ apart~\cite[Chap.~5 \S1]{Ghys-Harpe}. According to Theorem~21 in~\cite[Chap.~5]{Ghys-Harpe}, there are constants $H$ and $J$ depending only on $\delta$, $\lambda$ and $C$ such that any $(\lambda, C, J)$-local quasigeodesic remains at distance less than $H$ from any geodesic segment between its endpoints. Thus we can obtain the constants $H,J$ of the lemma by applying this result to $\lambda=1$ and $C$ depending only on $\delta$.

By hyperbolicity, the minimizing property of $P(x)$ implies that any $z\in [P(x), P(y)]$ satisfies $d(x,z) \geq d(x, P(x))+ d(P(x), z) - C$ for some $C$ depending only on $\delta$; indeed, this follows by comparing $[x, P(x)] \cup [P(x), P(y)]$ with a tripod in a tree, as can be done by Theorem~12 in~\cite[Chap.~2]{Ghys-Harpe}. We deduce that $[x, P(x)] \cup [P(x), P(y)]$ is a $(1,C)$-quasigeodesic. The same holds for $[P(x), P(y)]\cup [P(y), y]$. Thus, as long as $d(P(x), P(y)) \geq J$, the concatenation $[x, P(x)] \cup [P(x), P(y)]\cup [P(y), y]$ is a $(1, C, J)$-local quasigeodesic as required.
\end{proof}

\begin{proof}[Proof of Proposition~\ref{prop:hyp-two-uniform}]
We choose some point $x_0\in X$. Since $G$ is non-amenable, its action on $X$ is of general type by Proposition~\ref{prop:RelAmen-action}. By Lemma~\ref{lem:dual}, the pair $(\xi, \xi')$ is $G$-dual, so that there exists $(\gamma_n)$ in $G$ with $\gamma_n x_0  \to \xi$ and $\gamma_n^{-1} x_0  \to \xi'$. 

By the definition of convergence to $\xi$ and $\xi'$ in terms of Gromov products, every geodesic segment joining $\gamma_n^{-1} x_0 $ to $\gamma_n x_0 $ passes through a ball around $x_0$ of radius $R$ independent of $n$. Thus, $d(\gamma_n^{-1} x_0, \gamma_n x_0)$ is at least $2 d(\gamma_n x_0, x_0) - 2R$.

By Lemma~\ref{lem:hyp-isom}, this implies in particular that  $\gamma_n$ is hyperbolic for all sufficiently large $n$ with $|\gamma_n|_\infty\to\infty$ since $d(\gamma_n x_0, x_0)$ goes to infinity. We denote by $\xi_n^+$ and $\xi_n^-$ its attracting (resp. repelling) fixed point; that is, $\lim_{k\to \pm\infty} \gamma_n^{ k} x = \xi_n^\pm$ for all $x$, and $\xi_n^+ \neq\xi_n^-$.

Consider the union $A_n$ of all geodesic lines from $\xi_n^-$ to $\xi_n^+$; note that $\gamma_n$ preserves $A_n$. Any two such lines remain at distance less than an $X$-constant $D$ from each another; for instance, this follows by applying the tree approximation theorem (Theorem~12(ii) in~\cite[Chap.~2]{Ghys-Harpe}) to the union of the two geodesics.

We claim that the distance $d(x_0, A_n)$ is bounded independently of $n$.

Indeed, pick one of the lines $\ell$ defining $A_n$ and let $P$, $H$ and $J$ be as in Lemma~\ref{lem:proj-on-geod}. Hyperbolicity implies that all choices for $P$ are within some $D'$ from each other, where $D'$ depends only on the hyperbolicity constant. Since $|\gamma_n^2|_\infty= 2 |\gamma_n|_\infty \to\infty$ and hence also $|\gamma_n^2|\to\infty$, we can assume that $\gamma_n^2$ moves every point by more than $J+D+D'$. 
Note that the point $\gamma_n^2 P(\gamma_n^{-1} x_0)$ of $\gamma_n^2 \ell$ minimizes the distance from any point $z\in \gamma_n^2 \ell$ to $\gamma_n  x_0$ because $d(\gamma_n^2 P(\gamma_n^{-1} x_0), \gamma_n  x_0 ) = d( P(\gamma_n^{-1} x_0), \gamma_n^{-1}  x_0 ) \leq d(\gamma_n^{-2} z,  \gamma_n^{-1}  x_0 ) $ since $\gamma_n^{-2} z \in\ell$. Therefore, since $\ell$ and $\gamma_n^2 \ell$ remain within distance $D$, we see that $d(\gamma_n^2 P(\gamma_n^{-1} x_0), P(\gamma_n x_0)) \leq D+D'$ and hence $d(P(\gamma_n^{-1} x_0), P(\gamma_n x_0)) \geq J$. Thus Lemma~\ref{lem:proj-on-geod} applied to
$$S=[\gamma_n^{-1} x_0, P(\gamma_n^{-1} x_0) ] \cup [ P(\gamma_n^{-1} x_0), P(\gamma_n x_0)]\cup [P(\gamma_n x_0), \gamma_n x_0]$$
shows that $x_0$ is at distance at most $H+R$ of some $p\in S$. To reach our claim, it remains only to justify that $p$ does not belong to $[\gamma_n^{-1} x_0, P(\gamma_n^{-1} x_0) ] $ nor to $[P(\gamma_n x_0), \gamma_n x_0]$, and it suffices by symmetry to show the latter. The image of $[P(\gamma_n x_0), \gamma_n x_0]$ under $P$ has diameter bounded by some $X$-constant $E$; this follows e.g.\ from Proposition~2.1 in of~\cite[\S10]{CDP}). Therefore, using that $P$ is non-expanding up to another additive $X$-constant $E'$ (Corollary~2.2 loc.\ cit.), $p\in [P(\gamma_n x_0), \gamma_n x_0]$ would imply $d(P(x_0), P(\gamma_n x_0)) \leq E+E+H+R'$. Finally, comparing $\ell$ and $\gamma_n\ell$, we see that $d( P(\gamma_n x_0) , \gamma_n P(x_0))$ is bounded by an $X$-constant and hence we contradict
$$d(P(x_0),  \gamma_n P(x_0)) \geq  |\gamma_n| \geq |\gamma_n|_\infty\to\infty,$$
thus establishing the claim.

Since the assertion~\ref{it:bdd-tsl} holds by hypothesis, there exists  a hyperbolic element $\gamma'_n \in G$ with $|\gamma'_n|_\infty\leq K$ such that $\gamma'_n$ and $\gamma_n$ have the same attracting (resp. repelling) fixed point.

We claim that for any choice $b_n\in A_n$, the distance $d( \gamma'_n b_n, b_n)$ is bounded by $K+C$, where $C$ is an $X$-constant. Since $A_n$ is $\gamma'_n$-invariant and $D$-close to a geodesic, it suffices to prove this for \emph{some} choice $b_n\in A_n$, upon increasing $C$. Let $\delta$ be the hyperbolicity constant of $X$. There is $y_n\in X$ such that $d(\gamma' _n y_n, y_n) \leq |\gamma'_n| + \delta$, which implies $d(\gamma' _n y_n, y_n) \leq   |\gamma'_n|_\infty+ 17\delta  \leq K+ 17\delta$. Define $b_n=P(y_n)\in\ell\subseteq A_n$, with $\ell$ and $P$ as in the first claim. Noting once again that $\gamma'_n b_n$ is within an $X$-constant of $P(\gamma' y_n)$ and that $P$ is non-expanding up to an $X$-constant, we conclude that $d(\gamma'_n b_n, b_n)$ is bounded by $d(\gamma' _n y, y)$ plus an $X$-constant; the claim follows.

Since the first claim allows us to choose all $b_n$ within a bounded set, the second claim implies that the sequence $(\gamma'_n)$ is bounded in $G$, and subconverges to a hyperbolic isometry $\gamma$ with $|\gamma| \leq L=K+C$, and hence also $|\gamma|_\infty\leq L$. It remains to show that $\gamma$ fixes $\xi$ and $\xi'$.

Since $\gamma'_n$ fixes  $\xi^\pm_n$, it suffices to show that $\xi^-_n$ converges to $\xi'$ and $\xi^+_n$ to $\xi$; it is enough to show the latter. In terms of Gromov products, we need to show $\langle \xi^+_n, \xi \rangle_{x_0}\to\infty$. The hyperbolicity constant $\delta$ of $X$ satisfies
$$\langle \xi^+_n, \xi \rangle_{x_0} + \delta \geq \min\left( \langle \xi^+_n, \gamma_n x_0   \rangle_{x_0} ,  \langle \gamma_n x_0, \xi \rangle_{x_0}\right)$$
(see e.g.~\cite[2\S1]{CDP}). Thus, since $\gamma_n x_0 \to \xi$, it suffices to show $\langle \xi^+_n, \gamma_n x_0   \rangle_{x_0} \to\infty$. This follows by representing $\xi^+_n$ by a ray within $A_n$ because $d(\gamma_n x_0, A_n) = d(x_0, A_n)$ and this distance is bounded independently of $n$ by the first claim, while $d(x_0, \gamma_n x_0)$ goes to infinity.
\end{proof}

The following ingredient is again general.

\begin{lem}\label{lem:GO}
There is an $X$-constant $R$ such that for every $\xi\in \partial X$ and every $x\in X$ there is a geodesic $\sigma\colon \Rb \to X$ with $\sigma(+\infty)=\xi$ and passing within distance $R$ of $x$.
\end{lem}

\begin{proof}
Let $\lambda\colon\Rb\to X$ be some geodesic. By cocompactness, we can arrange that $\lambda(0)$ is at distance less than $D$ from $x$, where $D$ is the codiameter of $X$. Let further $\tau\colon\Rb_+\to X$ be geodesic ray pointing to $\xi$ with $\tau(0)=\lambda(0)$. We obtain two maps $\sigma_{\pm} \colon\Rb\to X$ by letting $\sigma_{\pm} (t)=\tau(t)$ for $t\geq 0$ and $\sigma_{\pm} (t)=\lambda(\pm t)$ for $t < 0$. If $X$ were a metric tree, one of $\sigma_-$ or $\sigma_+$ would already be geodesic. For general $X$, the tree approximation theorem (Theorem~12(ii) in~\cite[Chap.~2]{Ghys-Harpe}) gives an $X$-constant $C$ such that one of $\sigma_-$ or $\sigma_+$ is a $(1,C)$-quasi-geodesic. By Theorem~25(ii) in~\cite[Chap.~5]{Ghys-Harpe}, there is therefore a constant $H$ depending only on $X$ such that the corresponding $\sigma_{\pm}$ is at distance less than $H$ of a true geodesic $\sigma$. The statement follows with $R=D+H$.
\end{proof}

We turn to the proof of Theorem~\ref{thm:Ballistic}. 

\begin{proof}[Proof of Theorem~\ref{thm:Ballistic}]
The equivalence between~\ref{it:coco} and~\ref{it:bd-2-tr} follows from~\cite[Theorem~8.1]{CCMT}. The latter result also implies that~\ref{it:bd-2-tr} implies~\ref{it:all-ballistic}, whereas the implication from~\ref{it:all-ballistic} to~\ref{it:generic-ballistic} is obvious.  
We next observe that~\ref{it:generic-ballistic} implies~\ref{it:bdd-tsl} by Lemma~\ref{lem:hyp-bi-uniform}. It remains to prove that~\ref{it:bdd-tsl} implies~\ref{it:coco}.

Assuming~\ref{it:bdd-tsl}, we shall prove that the stabilizer $G_\xi$ acts cocompactly on $X$ for every $\xi \in \partial X$. This indeed implies that~\ref{it:coco} holds since $G_\xi$ is amenable by Proposition~\ref{prop:RelAmen-action}.

Let thus $x_0, x_1\in X$ be two arbitrary points.

Lemma~\ref{lem:GO} provides two geodesic lines $\sigma_0, \sigma_1$ with $\sigma_i$ passing within $R$ of $x_i$ and with $\sigma_i(+\infty)=\xi$. This common endpoint implies that there are points $y_i$ of $\sigma_i$ with $d(y_0, y_1) \leq R'$ for some $X$-constant $R'$; this follows e.g.\ from the tree approximation theorem (Theorem~12(ii) in~\cite[Chap.~2]{Ghys-Harpe}).

Proposition~\ref{prop:hyp-two-uniform} implies the existence of a hyperbolic isometry $\gamma_0$ of asymptotic displacement length $\leq L$, having the endpoints of $\sigma_0$ as its pairs of fixed points in $\partial X$. The orbit $\{\gamma_0^n x_0\}$ remains at distance at most $R+C$ of $\sigma_0$ for some $X$-constant $C$; choose $z_n$ on $\sigma_0$ with $d(\gamma_0^n x_0, z_n) \leq R+C$. As already observed in the proof of the second claim of that proposition, $\gamma_0$ will move any point of $\sigma_0$ by at most $L+C'$ for some $X$-constant $C'$. Since $d(\gamma_0 z_n, z_n+1)$ is bounded by an $X$-constant by the choice of $z_n$, it follows that the sequence $z_n$ has gaps bounded by $L+C''$ for some $X$-constant $C''$. In conclusion, there is $n_0$ with $d(\gamma_0^{n_0} x_0, y_0) \leq d(\gamma_0^{n_0} x_0, z_{n_0}) + d( z_{n_0}, y_0) \leq R+L+C+C''$.

We find likewise $\gamma_1$ and $n_1$ with $d(\gamma_1^{n_1} x_1, y_1) \leq R+L+C+C''$. In conclusion, $\gamma_1^{-n_1} \gamma_0^{n_0}$ is an element of $G_\xi$ which maps $x_0$ to within $R' +2( R+L+C+C'')$ of $x_1$. This witnesses that $G_\xi$ acts cocompactly on $X$ and concludes the proof of Theorem~\ref{thm:Ballistic}.
\end{proof}

Combining Propositions~\ref{prop:tdlc-reg-ell} and Theorem~\ref{thm:Ballistic}, we now deduce the following. 

\begin{cor}\label{cor:reg-ell-URS}
Let $G$ be a non-amenable hyperbolic locally compact group. Let $\mathcal F$ be  the stabilizer URS for the $G$-action on its Gromov boundary.

If $G$ does not have any cocompact amenable subgroup, then every $H \in \mathcal F$ is regionally elliptic. 
\end{cor}
\begin{proof}
By Theorem~\ref{thm:Ballistic}, for every continuity point $\eta \in \partial X$ of the stabilizer map, we have  $G_\eta^0 = G_\eta$.

The corollary is unaffected by replacing $G$ with its quotient by a compact normal subgroup, and such a quotient cannot be a virtually connected simple Lie group of rank one, since in that case parabolic subgroups are cocompact and amenable. Therefore, by  Proposition~\ref{prop:non-amen-hyper}, we can assume that $G$ is totally disconnected. 

Let now  $H \in \mathcal F$. Then $H \leq G_\xi$ for some $\xi \in \partial X$, and there is a sequence  $(g_n)$  in $G$ such that $(g_n G_\eta g_n^{-1})$ Chabauty converges to $H$. In view of Proposition~\ref{prop:tdlc-reg-ell}, it suffices to show $H \leq G^0_\xi$. Suppose for a contradition that this is not the case; then $X$ contains a hyperbolic isometry $t$. By Chabauty convergence, we can choose $t_n \in G_\eta$ such that $(g_n t_n g_n^{-1})$ converges to $t$. We then deduce from Lemma~\ref{lem:hyp-limit} that $t_n$ is hyperbolic for all large $n$. This is impossible since  $G_\eta = G_\eta^0$.
\end{proof}

\subsection{Roughly similar word metrics}

The goal of this section is to complete the proof of Theorem~\ref{thm:coctamen}.  In view of Proposition~\ref{prop:non-amen-hyper}, we shall  focus on totally disconnected groups.

Let $(X, d_X)$ and $(Y, d_Y)$ be metric spaces. Recall that a map $\varphi \colon X \to Y$ is  a \textbf{quasi-isometric embedding} if there exist constants $L >0$ and $C\geq 0$ such that
$$\frac 1 L d_X(x, x') - C \leq d_Y(\varphi(x), \varphi(x')) \leq L d_X(x, x') +C.$$
We say that $\varphi$ is a \textbf{roughly homothetic embedding} if there exist constants $L >0$ and $C\geq 0$ such that
$$ L d_X(x, x') - C \leq d_Y(\varphi(x), \varphi(x')) \leq L d_X(x, x') +C.$$
Thus Corollary~\ref{cor:Z-action-homothety} states precisely that every orbit of a hyperbolic isometry of a hyperbolic geodesic metric space is a rough embedding of $\Zb$ into that space, where $\mathbf Z$ is endowed with the Euclidean metric.

A quasi-isometric (resp. roughly homothetic) embedding  $\varphi \colon X  \to Y$ is called a \textbf{quasi-isometry} (resp. a \textbf{rough similarity}) if $Y$ is contained in a bounded neighbourhood of $\varphi(X)$.  

It is easy to see that any two word metrics on $\mathbf Z$ given by finite generating sets are roughly similar (this can also be considered as a trivial case of Corollary~\ref{cor:Z-action-homothety}). This contrasts with the free group $F_2$, which has numerous pairs of finite generating sets giving rise to word metrics that are quasi-isometric, but not roughly similar.

Let now $G$ be a compactly generated tdlc group and $U \leq G$ be a compact open subgroup. Given a compact generating set $\Sigma \subset G$, the graph with vertex set $G/U$ and edge set defined by $gU \sim hU \Leftrightarrow h^{-1} g \in U\Sigma U \cup U \Sigma^{-1} U$, is connected, locally finite, and preserved by the natural $G$-action. It is called the \textbf{Cayley--Abels graph} associated with the pair $(U, \Sigma)$. 

\begin{prop}\label{prop:RoughSimilarity}
For a non-amenable hyperbolic tdlc group $G$, the following assertions are equivalent. 
\begin{enumerate}[label=(\roman*)]
\item $G$ has a cocompact amenable subgroup. 

\item Considering any Cayley--Abels graph $\mathcal G$ for $G$ and endowing $G$ with the word metric associated to any compact generating set, each orbit map $G\to V\mathcal G$ is rough similarity.

\item For some compact open subgroup $U \leq G$, the identity on $G/U$ is a rough similary between the Cayley--Abels graphs associated with $(U, \Sigma_1)$ and $(U, \Sigma_2)$, for any pair of compact generating sets $\Sigma_1, \Sigma_2$. 
\end{enumerate}
\end{prop}
\begin{proof}
(i) $\Rightarrow$ (ii). Let $x \in  V\mathcal G$. Since the orbit map $g \mapsto gx$ is $G$-equivariant, it suffices to show that there  exist constants $A, C >0$ such that
$$A |g| - C \leq d_{\mathcal G}(x, gx) \leq A|g| + C \kern3mm (\forall\, g\in G)$$
where $|g|$ is the length of $g$ in the chosen word metric. We invoke~\cite[Theorem~8.1]{CCMT}, which ensures that up to a compact kernel $G$ is a tree automorphism group acting $2$-transitively on the boundary of that tree. In particular, there is a compact subset $K\subseteq G$ and an element $a\in G$ acting as a hyperbolic isometry of that tree such that the Cartan-like decomposition $G=K \{a^n : n \in \Nb\} K$ holds. Given $g\in G$, we denote by $n_g$ some choice of integer with $g\in K a^{n_g} K$.

We recall that the Cayley graph associated to the generating set defining $|\cdot|$, although  unlike $\mathcal G$ it is far from locally finite, is a hyperbolic geodesic metric space (compare~\cite[\S2]{CCMT}). Moreover, $a$ is a hyperbolic isometry of this space because its displacement length does not vanish. Since Corollary~\ref{cor:Z-action-homothety} is valid in that generality, the difference $|a^n| - n A_0$ remains bounded independently of $n$, where $A_0$ is the asymptotic displacement length of $a$ in that Cayley graph.

Since the word length is bounded over $K$, there is a constant $C_0$ with
$$|a^{n_g}| - C_0 \leq |g|  \leq |a^{n_g}| + C_0$$
for all $g$, and hence
$$n_g A_0 - C_1 \leq |g|  \leq n_g A_0 + C_1$$
for some $C_1$. Similarly, since $K$ has bounded orbits in ${\mathcal G}$, let $C_2$ be a bound for $d_{\mathcal G}(k x, x)$ over $k\in K$. Then, writing $g=k_1 a^{n_g} k_2$, the triangle inequality gives
$$d(x, a^{n_g}x) - 2C_2 \leq d(x, gx) \leq d(x, a^{n_g} x) + 2C_2$$
because $d(x, gx) = d(k_1^{-1}x, a^{n_g} k_2 x)$. We apply again Corollary~\ref{cor:Z-action-homothety} but to the action of $a$ on ${\mathcal G}$ and deduce that the difference $d(x, a^{n_g} x) - n_g A_1$ remains bounded, where $A_1$ is the asymptotic displacement length of $a$ in ${\mathcal G}$. Putting everything together, we obtain the required conclusion with $A=A_1/A_0$.

\medskip \noindent (ii) $\Rightarrow$ (iii). Obvious. 

\medskip \noindent (iii) $\Rightarrow$ (i). Fix a compact generating set $\Sigma$ and let $\mathcal G_0$ denote the Cayley--Abels graph for $G$ with respect to $(U, \Sigma)$. The distance function on $V\mathcal G_0 = G/U$ given by that graph structure is denoted by $d_0$.

For each geodesic line $\tau \colon \mathbf Z \to V\mathcal G_0$, we choose a map $P_\tau \colon V\mathcal G_0 \to \tau(\mathbf Z)$ such that for all $v \in V\mathcal G_0 = G/U$, we have $d_0(v, P_\tau(v)) = \min \{d_0(v, \tau(n)) \mid n \in \mathbf Z\}$. Thus $P_\tau$ is a nearest-point-projection on the geodesic line $\tau$. By~\cite[Proposition~10.2.1]{CDP}, there exists an $X$-constant $C$ such that 
$$
d_0(P_\tau(x), P_\tau(y)) \leq  
\max \{C, C+ d_0(x, y)-d_0(x, P_\tau(x))-d_0(y, P_\tau(y))\}
$$
for all vertices $x, y$.

Let $a \in G$ be a hyperbolic element. Denote by $\xi, \xi' \in \partial  G$ its fixed points at infinity. Fix a geodesic line $\sigma \colon \mathbf Z \to V \mathcal G_0$ with endpoints $(\xi, \xi')$. We also consider the union $A$ of all geodesic lines from $\xi$ to $\xi'$. As in the proof of Proposition~\ref{prop:hyp-two-uniform}, any two such lines are at distance less than some $X$-constant $D$ from one another. 

\medskip
We claim that for  any geodesic line  $\tau \colon \mathbf Z \to V \mathcal G_0$  and any $L>0$, there exists $g_L \in G$ such that the diameter of $P_\tau(g_L(A))$ is larger than $L$. 

\medskip

Suppose for a contradiction that the claim fails for some geodesic line $\tau$. Hence there exists $L_0 >0$ such that, for any $g \in G$, the diameter of $P_\tau(g(A))$   is at most $L_0$. Upon enlarging $L_0$, we may assume that $L_0 \geq C +1$.

Let now  $L>L_0$ be an even integer and choose an element $h_L \in G$ be mapping $\sigma(0)$ to $\sigma(L)$. 
Set $\Sigma_L = \Sigma \cup \{h_L\}$, let  $\mathcal G_L$ be the Cayley--Abels graph for $G$ with respect to $(U, \Sigma_L)$ and $d_L$ be the associated distance function on $G/U$. By construction we have $d_L(x, y) \leq d_0(x, y)$ for all $x, y \in G/U$ since every edge of $\mathcal G_0$ is also an edge of $\mathcal G_L$.  Moreover, every edge of $\mathcal G_L$ that is not in $\mathcal G_0$ is of the form $\{g\sigma(0), g\sigma(L)\}$ for some $g \in G$. 
Following the terminology introduced in~\cite{BF09}, every $G$-translate of the geodesic segment $\sigma|_{[0, L]}$ in $\mathcal G_0$ is called an \textbf{expressway}.  

By the hypothesis (iii), there exist constants $A_L, C_L$ such that
$$A_L d_L(x, y) -C_L \leq d_0(x, y) \leq A_L d_L(x, y) + C_L.$$
Apply that inequality with $x = \sigma(0)$ and $y = a^n(x)$, divide by $n$ and let $n$ tend to infinity. We deduce that $A_L = |a|_{\infty,0}/|a|_{\infty, L}$, where $|a|_{\infty,0}$ and $|a|_{\infty, L}$ denote the asymptotic displacement length of $a$ with respect to the metrics $d_0$ and $d_L$. Observe that any point of the $\langle a \rangle$-orbit of $\sigma(0)$ remains $D$-close to $\sigma(\mathbf Z)$ in the metric $d_0$. Since $d_L \leq d_0$, we deduce from the definition of the metric $d_L$ that  the asymptotic displacement length $|a|_{\infty, L}$ tends to $0$ as $L$ tends to infinity. We may therefore assume, upon taking $L$  large enough,  that $A_L >  2 L_0$.

Let now  
$N > C_L(1+\frac{A_L }{L_0})$ be an integer and consider a $\mathcal G_L$-geodesic segment  
$$(v_0 = \tau(0), v_1, \dots, v_M = \tau(N))$$
joining $\tau(0)$ to $\tau(N)$. We have $M \leq N$ since $d_L(x, y) \leq d_0(x, y)$ for all $x, y$. 
For each pair of vertices $v_i, v_{i+1}$, either $\{v_i, v_{i+1}\}$ is an edge of $\mathcal G_0$, or there exists an expressway with endpoints $\{v_i, v_{i+1}\}$. 
By concatenating those edges and expressways, we obtain a continuous path $[v_0, v_1] \cup [v_1, v_2] \cup \dots$ in the graph $\mathcal G_0$. By the definition of $L_0$, we know that for each $i$ such that $[v_i, v_{i+1}]$ is an expressway, we have $d(P_\tau(v_i), P_\tau(v_{i+1})) \leq L_0$. If $[v_i, v_{i+1}]$ is an edge, then by the defining property of the constant $C$,  we have $d(P_\tau(v_i), P_\tau(v_{i+1})) \leq 1+ C \leq L_0$. 

Since $(v_0 = P_\tau(v_0), P_\tau(v_1), \dots, P_\tau(v_M) = v_M)$ defines a path from $v_0$ to $v_M$, we deduce from the triangle inequality that 
$$N = d_0(\tau(0), \tau(N)) = d_0(v_0, v_M) \leq \sum_{i=1}^M d_0(P_\tau(v_{i-1}), P_\tau(v_i)) \leq L_0 M.$$
On the other hand, by the hypothesis (iii), we have 
$$A_L M -C_L = A_Ld_L(v_0, v_M)  - C_L \leq d_0(\tau(0), \tau(N)) = N.$$ 
Since  $A_L >  2 L_0$, we obtain
$$2L_0 M -C_L < L_0 M,$$
so that $ M < \frac{C_L}{L_0}$. The hypothesis (iii) also implies that 
$$ N = d_0(\tau(0), \tau(N)) \leq A_Ld_L(v_0, v_M)  + C_L = A_L M +C_L,$$ 
so that 
$$N < \frac{A_L C_L}{L_0} +C_L.$$
This contradicts the choice of $N$, thereby establishing the claim. 

\medskip
Let now $b \in G$ be a hyperbolic element with attracting and repelling fixed points $\eta_+, \eta_- \in \partial G$. Fix a 
geodesic line  $\tau \colon \mathbf Z \to V \mathcal G_0$ 
 with  $\eta_+ = \tau(\infty)$ and $\eta_- = \tau(-\infty)$. 
For each even integer $L>0$, the claim ensures the existence of an element  $g_L \in G$ be such that $P_\tau(g_L(A))$ has  diameter strictly greater than~$2C+ 2D+2L$. Let $x, y \in A$ with $d_0(P_\tau(g_L(x)), P_\tau(g_L(y)))>2C+ 2D+2L$. Pick $x', y' \in \sigma(\mathbf Z)$ with $d_0(x, x') \leq D$ and $d_0(y, y') \leq D$. In particular we have $d_0(P_\tau(g_L(x')), P_\tau(g_L(y')))> 2L$. We now apply Lemma~\ref{lem:proj-on-geod} to obtain constants $H,J$ and henceforth consider only $L$ large enough to have $L\geq 2\max\{C, H, J\}$. Thus, the piecewise geodesic segment 
$$[g_L(x'), P_\tau(g_L(x'))] \cup [P_\tau(g_L(x')), P_\tau(g_L(y'))] \cup [P_\tau(g_L(y')), g_L(y')]$$ 
lies in an $H$-neighbourhood of the subsegment of the geodesic line $g_L(\sigma)$ joining $g_L(x')$ to $g_L(y')$. Let $n_L$ be an integer such that $d_0(P_\tau(g_L(x')), g_L(\sigma(n_L))) \leq H$, and $n'_L \geq 0$ be an integer such that $d_0(P_\tau(g_L(y')), g_L(\sigma(n_L+n'_L))) \leq H$. By the definition of $n_L$ and $n'_L$, we  have $d_0(g_L(\sigma(n_L)), P_\tau(g_L(\sigma(n_L))))\leq H$ and $d_0(g_L(\sigma(n_L+n'_L)), P_\tau(g_L(\sigma(n_L+n'_L))))\leq H$. Therefore,  
\begin{align*}
	d_0(g_L(\sigma(n_L)), g_L(\sigma(n_L+n'_L))) & \geq  d_0(P_\tau(g_L(x')), P_\tau(g_L(y')))-2H\\
	& > 2L-2H\\
	& > L.
\end{align*}
This ensures that $n'_L > L$. In particular, the Hausdorff distance between the geodesic segments $g_L \circ \sigma|_{[n_L, n_L + L]}$ and $[P_\tau(g_L(x')), P_\tau(g_L(y'))]$ is bounded by an $X$-constant. 
Therefore it lies 
in a bounded neighbourhood of $\tau(\mathbf Z)$, where the bound is independent of $L$.

Using Corollary~\ref{cor:Z-action-homothety}, we may find  suitable integers $s, t$ such that the element $h_L = b^s g_L a^t$ maps the pair $\{\sigma(-L/2), \sigma(L/2)\}$ at uniformly bounded distance from $\{\tau(-L/2), \tau(L/2)\}$. In particular $d_0(h_L(\sigma(0)), \tau(0))$ is bounded independently of $L$. Therefore  the sequence $(h_L)$ is bounded in $G$, and thus subconverges to an element $h \in G$. By construction $h$ maps the pair $\{\xi, \xi'\}$ to the pair $\{\eta_+, \eta_-\}$.  

Notice that $hah^{-1}$ is a hyperbolic element with the same asymptotic displacement length as $a$, and with the same pair of fixed points in $\partial X$ as $b$. Fixing $a$ and letting $b$ vary over the collection of all hyperbolic isometries in $G$, we deduce that the assertion~\ref{it:bdd-tsl} from Theorem~\ref{thm:Ballistic} is satisfied, which therefore yields the required conclusion. 
\end{proof}

\begin{proof}[Proof of Theorem~\ref{thm:coctamen}]
The equivalences between the assertions~\ref{it:coco}--\ref{it:bdd-tsl} are established by Theorem~\ref{thm:Ballistic}. Their equivalence with~\ref{it:metric} follows from Proposition~\ref{prop:RoughSimilarity} in case $G$ is totally disconnected. 

In general, since $G$ is non-amenable it is of general type, we see from~\cite[Proposition~5.10]{CCMT} that, after dividing out a compact normal subgroup, it is either a virtually connected rank one simple Lie group, or a totally disconnected group. Since every connected Lie group has a cocompact solvable subgroup, it suffices to show that a virtually connected rank one simple Lie group satisfies~\ref{it:metric}. This follows from the $KAK$-decomposition as in the first part of the proof of Proposition~\ref{prop:RoughSimilarity}; we omit the details.
\end{proof}

\section{Boundary representations of hyperbolic groups and the type~I property}

The goal of this section is to complete the proofs of the remaining results from the introduction. 

\subsection{Quasi-regular representations defined by cocompact subgroups}
 In the special case where $P$ is unimodular, the following result can be extracted from the proof of~\cite[Theorem~9.2.2]{DE}. 

\begin{prop}\label{prop:CCR-coco}
Let $G$ be a tdlc group and $P$ be a closed cocompact subgroup. Then every element   of $C^*_{\lambda_{G/P}}(G)$ is a compact operator. In particular, it splits as a direct sum of irreducible CCR representations. 
\end{prop}
\begin{proof} 
Let $U$ be a compact open subgroup of $G$ and $\nu$ be a quasi-invariant probability measure on $G/P$. By Lemmas~\ref{lem:conv->equiv-meas} and~\ref{lem:equiv-meas->equiv-koop}, we may replace $\nu$ by $\mu *\nu$, where $\mu$ is a $U$-invariant probability measure on $G$, so as to ensure that  $\nu$ is $U$-invariant. For any open subgroup $V \leq U$, the dimension of the space of $V$-invariant vectors in $L^2(G/P, \nu)$ equals the cardinality $|V\backslash G/P|$ of the set of double cosets modulo $(V, P)$. Since $P$ is cocompact and $V$ is open, that number is finite. It follows that the orthogonal projection $p_V$ on the space of $V$-invariant vectors, is a compact operator. Observe moreover that $p_V = \frac 1 {\mu_G(V)} \lambda_{G/P}(\mathbf 1_V)$, where $\mu_G$ denotes a left Haar measure on $G$ and $\mathbf 1_V$ is the characteristic function of $V$. Hence the operators $(p_V)$, where $V$ runs over the compact open subgroups of $U$, form an approximate unit in $C^*_{\lambda_{G/P}}(G)$, consisting entirely of  compact operators. This implies that every element of $C^*_{\lambda_{G/P}}(G)$ is a compact operator. 
The required conclusions then follow from~\cite[Theorem~1.4.4]{Arveson}.
\end{proof}

In~\cite[Theorem A]{Raum-cocpt} Raum proved that a locally compact second countable unimodular $C^*$-simple group $G$ does not admit any cocompact amenable closed subgroups. Using the above proposition we give a simple proof of this fact for general locally compact groups $G$, dropping the assumptions of second countability and unimodularity.

We need the following fact which is certainly known to the experts. We include a proof for the sake of completeness.

Recall that a $C^*$-algebra is called {\bf elementary} if it is isomorphic to the $C^*$-algebra of compact operators on some Hilbert space.

\begin{lem}\label{lem:elem-C*->triv}
If $G$ is a non-trivial locally compact group, then the $C^*$-algebra $C^*_r(G)$ is non-elementary.
\end{lem}
\begin{proof}
Suppose $C^*_r(G)$ is elementary. Then it follows from~\cite[Proposition 4.3]{Runde} that $L^1(G)$ has a unique $C^*$-norm. In particular, we have $C^*(G)=C^*_r(G)$, which implies $C^*_r(G)$ admits a character. Since the $C^*$-algebra of compact operators is simple, it follows $G$ is trivial.
\end{proof}

\begin{thm}\label{thm:non-C*-simple}
Let $G$ be a non-trivial locally compact group containing a cocompact amenable closed subgroup. Then $C^*_r(G)$ is not simple.
\end{thm}
\begin{proof}
Assume $G$ contains a cocompact amenable closed subgroup $P$, and $C^*_r(G)$ is simple. By~\cite[Theorem 6.1]{Raum-C*-simple}, $G$ is totally disconnected. Since $P$ is amenable, $\lambda_{G}$ weakly contains $\lambda_{G/P}$, therefore they are weakly equivalent by $C^*$-simplicity. Then Proposition~\ref{prop:CCR-coco} implies $C^*_r(G)$ is CCR, and since it is also simple, it follows it is elementary. Hence, $G$ is trivial by Lemma~\ref{lem:elem-C*->triv}.
\end{proof}

\subsection{Weak equivalence of boundary representations}

\begin{proof}[Proof of Theorem~\ref{thm:weak-equiv-bd-rep}]
Let us first observe that the required assertion holds in the special case where $G$ has a cocompact amenable subgroup $P$. Then, by Theorem~\ref{thm:Ballistic}, the $G$-action on $\partial G$ is transitive, so that there is a unique  $G$-invariant measure class on $\partial G$, see~\cite[Theorem~B.1.4]{BHV} and~\cite[Lemma~8.1]{CCMT}. Hence, by Lemma~\ref{lem:equiv-meas->equiv-koop}, any two boundary representations are unitarily equivalent. In particular, they are weakly equivalent. 

We next invoke  Proposition~\ref{prop:non-amen-hyper} and denote by $W$ the compact normal subgroup afforded in that way. Since $W$ acts trivially on the Gromov boundary $\partial G$, and is thus contained in the kernel of every boundary unitary representation, there is no loss of generality in assuming that $W$ is trivial. Hence, either $G$ is a virtually connected rank one simple Lie group, or $G$ is totally disconnected. 

In the former case, the group $G$ has a cocompact amenable subgroup, and the required conclusion follows by the first paragraph. 
We assume henceforth that $G$ is totally disconnected and that $G$ does not have any cocompact amenable subgroup.  Therefore, Corollary~\ref{cor:reg-ell-URS} ensures that some point $\xi \in \partial G$ has a regionally elliptic stabilizer. Since the $G$-action on $\partial G$ is topologically amenable (see~\cite{Adams},~\cite{Kaimanovich}), all points have an amenable stabilizer, and the $G$-action on $(\partial G, \nu)$ is Zimmer-amenable with respect to quasi-invariant probability measure $\nu$ (see Proposition~3.3.5 in~\cite{ADR}). Therefore the hypotheses of Corollary~\ref{cor:amenable-actions} are fulfilled. The required conclusion follows. 
\end{proof}

\subsection{Type~I hyperbolic groups}

It remains to prove Theorems~\ref{thm:main} and~\ref{thm:GCR}. As announced in the introduction, we shall rely on Garncarek's work~\cite{Garncarek}, which concerns the boundary representations of discrete hyperbolic groups associated with Patterson--Sullivan measures. 

For our purposes, it is sufficient to consider those measures in the setting of hyperbolic locally finite graphs with a vertex transitive automorphism group. Given such a graph $X$, the Patterson--Sullivan construction (developed by M.~Coornaert~\cite{Coornaert} in the general hyperbolic setting) yields a canonical measure class on the Gromov boundary $\partial X$, which is invariant under the full automorphism group $\Aut(X)$. The probability measures in that class can be defined by Hausdorff measures associated with visual metrics or, alternatively, as weak*-limits of normalized counting measures on balls around a fixed vertex of the graph $X$. The Koopman representation defined by any probability measure in the Patterson--Sullivan class is called the \textbf{PS-representation} associated with $X$. We denote it by $\kappa_X$. It is well defined up to equivalence (see Lemma~\ref{lem:equiv-meas->equiv-koop}). Its domain is  the full automorphism group $\Aut(X)$.

In order to clarify how Garncarek's work comes into play, we introduce the following two conditions on a hyperbolic tdlc groups $G$ and compact open subgroup $U \leq G$:
\begin{description}
\item[(G1)] For any Cayley--Abels graph $X$ on $G/U$, the PS-representation  $\kappa_X$ is an irreducible representation of $G$. 

\item[(G2)] For any two Cayley--Abels graphs $X_1, X_2$ on $G/U$, if the PS-representations  $\kappa_{X_1}, \kappa_{X_2}$ are unitarily equivalent, then the identity on $G/U$, viewed as a map from $VX_1$ to $VX_2$, is a rough similarity. 
\end{description}
	
The main results of~\cite{Garncarek} state notably that every non-elementary \emph{discrete} hyperbolic group satisfies (G1) and (G2). Applying this to a uniform lattice in the non-discrete case, we obtain the following.

\begin{thm}[Garncarek] \label{thm:Garncarek}
Let $G$ be a non-amenable hyperbolic tdlc group. If $G$ has a cocompact lattice, then for every compact open subgroup $U \leq G$, the conditions (G1) and (G2) are satisfied.
\end{thm}

\begin{proof}
Let $U \leq G$ be any compact open subgroup. Let also $\Gamma \leq G$  be a cocompact lattice. For any Cayley--Abels graph $X$ on $G/U$, the restriction of the PS-representation $\kappa_X$ to $\Gamma$ is irreducible by~\cite[Theorem~6.2]{Garncarek}. In particular, it is irreducible as a representation of $G$, so (G1) holds. 

Given two Cayley--Abels graphs $X_1, X_2$ on $G/U$, if the PS-representations  $\kappa_{X_1}, \kappa_{X_2}$ are unitarily equivalent, then their restrictions to $\Gamma$ are equivalent. It then follows from (the proof of)~\cite[Theorem~7.4]{Garncarek} that the identity on $\Gamma$ defines a rough similarity of $(\Gamma, d_1)$ to $(\Gamma, d_2)$, where $d_i$ is the pseudo-metric on $\Gamma$ induced by the orbit map $\Gamma \to VX_i = G/U : \gamma \mapsto \gamma U$. Let $V'_i$ be the image of that map. 

Since $\Gamma$ is cocompact in $G$, we have $G = \Gamma K$ for some compact subset $K$ of $G$. Since $K$ is covered by finitely many left $U$-cosets, it follows that $\Gamma$ acts with finitely many orbits on $VX_i$. In particular, by choosing for each $v \in VX_i$ a vertex $p_i(v) \in V'_i$ that is at minimal distance from $v$, we obtain a map $p_i \colon VX_i \to V'_i$ such that the distance from $v$ to $p_i(v)$ is uniformly bounded. 

By the very definition, the orbit map $\Gamma \to VX_i$ defines a rough similarity of $(\Gamma, d_i)$ to $V'_i$. It follows that the identity on $\Gamma U/U$, viewed as a map $\varphi' \colon  V'_1  \to V'_2$, is a rough similarity.  By the triangle inequality, it follows that the map $\varphi \colon VX_1 \to VX_2 : v \mapsto \varphi'(p_1(v))$ is a rough similarity. Let $\psi\colon VX_1 \to VX_2$ be the map defined by the identity on $G/U$. By construction the maps $\varphi$ and $\psi$ coincide on $V'_1$. Since $\psi$ is a quasi-isometry, it follows that the distance from $\psi(v)$ to $\psi(p_1(v))$ is uniformly bounded. Therefore, the maps $\varphi$ and $\psi$ are within bounded distance of each other. Since $\varphi$ is a rough similarity, it follows that $\psi$ has also this property. Thus (G2) holds.
\end{proof}

To complete the proofs of the  results announced in the introduction, it remains to establish the following.

\begin{thm}\label{thm:main-tech}
Let $G$ be a non-amenable hyperbolic tdlc group and $U \leq G$ be a compact open subgroup satisfying (G1) and (G2). If for some boundary representation $\kappa$ of $G$, the C*-algebra $C^*_\kappa(G)$ contains a non-zero CCR two-sided ideal, then  $G$ has a cocompact amenable subgroup. 	
\end{thm}
\begin{proof}	
Set $A = C^*_\kappa(G)$ and let $I$ be a non-zero CCR two-sided ideal in $A$.  
By Theorem~\ref{thm:weak-equiv-bd-rep}, any two boundary representations of $G$ are weakly equivalent. Thus, it follows that for every boundary representation $\pi$ of $G$, we have $C^*_\pi(G) \cong A$.

Let now $X$ be any Cayley--Abels graph on $G/U$. By (G1), the representation  $\kappa_X$ is irreducible. Since  $C^*_{\kappa_X}(G) \cong A$, we deduce that $C^*_{\kappa_X}(G)$ contains a non-zero CCR two-sided ideal, which acts  irreducibly on $\mathcal H_{\kappa_X}$ by~\cite[Theorem~1.3.4]{Arveson}. Since that ideal is CCR,  it entirely consists of compact operators. 
Since this is valied for any Cayley--Abels graph $X$, we deduce from~\cite[Corollary~4.1.10]{Dixmier} that, for any two Cayley--Abels graphs $X_1, X_2$ on $G/U$, the representations $\kappa_{X_1}$ and $\kappa_{X_2}$ are unitary equivalent.
Since the pair $(G, U)$ satisfies (G2), we infer that the identity on $G/U$, viewed as a map from $VX_1$ to $VX_2$, is a rough similarity. Proposition~\ref{prop:RoughSimilarity} now ensures that $G$ has a cocompact amenable subgroup.
\end{proof}

\begin{proof}[Proof of Theorem~\ref{thm:GCR}]
The implications~\ref{it:all-compact-op} $\Rightarrow$~\ref{it:CCR} $\Rightarrow$~\ref{it:GCR} $\Rightarrow$~\ref{it:CCR-ideal} are obvious. If~\ref{it:amen-coco} holds, we see as in the proof of Theorem~\ref{thm:weak-equiv-bd-rep} that any two boundary representations are unitarily equivalent (and not only weakly equivalent). Moreover, by~\cite[Theorem~8.1]{CCMT}, we may assume that $G$ is either a virtually connected rank one simple Lie group, or a closed subgroup of the automorphism group of a non-elementary locally finite tree, acting without inversion, with exactly $2$ orbits of vertices, and acting $2$-transitively on the set of ends. In the former case, the assertion~\ref{it:CCR} follows from known results on unitary representations theory of simple Lie groups. Indeed $\kappa$ is the representation unitarily induced from the trivial representation of a parabolic subgroup and therefore it is admissible, see Proposition~8.4 in~\cite{Knapp_rep}. For the fact that an admissible representation maps every element of the group C*-algebra to a compact operator,   we refer e.g.\ to the proof of (i) and (ii) in Lemma~15.5.1 in~\cite{Dixmier} or to the proof of Proposition~6.E.11 in~\cite{BekkaHarpe}. In the totally disconnected case, the assertion~\ref{it:all-compact-op} follows from Proposition~\ref{prop:CCR-coco}.

Assume finally that~\ref{it:CCR-ideal} holds; we shall show~\ref{it:amen-coco}. We may assume that $G$ is non-amenable. By Proposition~\ref{prop:non-amen-hyper} we may further assume that $G$ is totally disconnected, since every connected Lie group has a cocompact solvable subgroup. The required conclusion then follows from Theorems~\ref{thm:Garncarek} and~\ref{thm:main-tech}.
\end{proof}

\begin{proof}[Proof of Theorem~\ref{thm:main}]
If $G$ is type~I, then since $G$ is $\sigma$-compact, every unitary representation of $G$ is GCR (see~\cite[Theorem~8.F.3]{BekkaHarpe}). If $G$ is amenable, there is nothing to prove; otherwise, the theorem follows from the implication~\ref{it:GCR}~$\Rightarrow$~\ref{it:amen-coco} in Theorem~\ref{thm:GCR}.
\end{proof}

\begin{rmk}\label{rem:Garn}
We strongly believe that Theorems~\ref{thm:main} and~\ref{thm:GCR} hold without the requirement that $G$ contains a uniform lattice. This would follow from Theorem~\ref{thm:main-tech} by the same reasoning as  above, provided one shows that  for every non-amenable hyperbolic tdlc group $G$, there is a compact open subgroup $U$ such that (G1) and (G2) are satisfied. Whether this non-discrete generalization of Garncarek's results~\cite{Garncarek} holds is a highly interesting question in its own. This question can be envisaged in the context of a wide-ranging conjecture of Bader and Muchnik~\cite{BadMuch}. That conjecture predicts that the   Koopman representation associated with the Poisson boundary of  the random walk defined by a spread-out probability measure $\mu$ on an arbitrary locally compact is irreducible. (Recent results by Bj\"orklund--Hartman--Oppelmayer~\cite{BjoHarOpp} suggest that the measure $\mu$ should be assumed to be symmetric  in the Bader--Muchnik conjecture.)
\end{rmk}

\begin{proof}[Proof of Corollary~\ref{cor:hyper-coco-amen}]
If $G$ is compact or $2$-ended, then~\ref{it:hyper-coco-amen:elem} holds by~\cite[Proposition~5.6]{CCMT}. We assume henceforth that this is not the case, so $G$ is non-elementary. The existence of a lattice implies that $G$ is unimodular, hence $G$ is non-amenable by~\cite[Theorem~7.3]{CCMT}. It then follows from~\cite[Theorem~D]{CCMT} that~\ref{it:hyper-coco-amen:Lie} or~\ref{it:hyper-coco-amen:tree} holds. 
\end{proof}
	
\begin{proof}[Proof of Corollary~\ref{cor:HR}]
If $G$ were compactly generated, then it would act cocompactly on $T$ by \cite[Lemma~2.4]{CDM}, hence it would be Gromov hyperbolic. If moreover $G$ were unimodular, then it would contain a uniform lattice by the main result of~\cite{BaKu}. Hence, in that situation, we may invoke Corollary~\ref{cor:hyper-coco-amen}. Although the tree given by Corollary~\ref{cor:hyper-coco-amen} need not be $T$ itself, its boundary is canonically identified with $\partial T$ as a $G$-space by cocompactness of the $G$-action on $T$. The required conclusion follows.

We now present two different approaches to finish the proof. The first one consists in invoking \cite[Theorem~A]{HoudayerRaum}, which ensures that, after discarding vertices of degree~$2$ of $T$, the action of every vertex stabilizer in $G$ is $2$-transitive on the neighbouring edges. It is easy to deduce that $G$ acts edge-transitively on $T$, that $G$ is compactly generated, and that $G$ is unimodular (because it is generated by compact subgroups). The hypotheses required by the argument of the first paragraph above are fulfilled, and we are done. 

We may alternatively avoid invoking the full strength of \cite[Theorem~A]{HoudayerRaum}, as follows.  
Let $G^+$ be the open subgroup of index at most~$2$ preserving the canonical bipartition of $T$. Then $G^+$ acts without inversion. If the quotient graph $G^+ \backslash T$ is not a tree, then $G$ cannot be a type~I group by~\cite[Proposition~4.1 and 4.2]{HoudayerRaum}. Since $G$, hence also $G^+$, is of type~I (see \cite[Proposition~2.4]{Kallman73}), we infer that   $G^+$ is the fundamental group of a tree of profinite groups. In particular it is generated by compact subgroups, hence it is unimodular. Let us write $G^+$ as an ascending union of compactly generated open subgroups $G^+ = \bigcup_n H_n$. Without loss of generality we assume that $H_0$ is not compact. For each $n$ let $T_n \subset T$ be a minimal $H_n$-invariant subtree of $T$. By \cite[Lemma~2.4]{CDM}, the group $H_n$ acts cocompactly on $T_n$, so that $H_n$ is hyperbolic. As an open subgroup of $G^+$, the group $H_n$ is unimodular. Applying Corollary~\ref{cor:hyper-coco-amen} to $H_n$, we invoke that  the $H_n$-action is $2$-transitive on the ends of $T_n$. Let now $v \in V(T_0)$ be a vertex and set $K = G^+_v$. For $n$ large enough we have $K \leq H_n$. Since $T_0$ is contained in $T_n$, we have $H_n = K\langle a_n \rangle K$ for a suitable hyperbolic element $a_n \in H_n$. It follows that $K$ is a maximal subgroup of $H_n$. Similarly $K$ is maximal in $H_m$ for all $m \geq n$. Since $K \leq H_n \leq H_{m}$, we infer that $H_n = H_m$ for all $m \geq n$. Hence $G^+ = H_n$, which implies that $G^+$ is compactly generated. This confirms again that the hypotheses required by the argument of the first paragraph above are fulfilled.
\end{proof}

\begin{rmk}
We note that in proving Theorem~\ref{thm:main} we only used that boundary representations of $G$ are GCR. This is formally weaker than type~I property, which is equivalent to all representations of $G$ being GCR. In fact, for every almost invariant measure $\nu$ on $\partial G$, the $G$-space $(\partial G, \nu)$ is amenable by~\cite[Theorem~6.8]{Adams}, so that  every boundary representation $\kappa_\nu$ of $G$ is weakly contained in the regular representation $\lambda_G$ of $G$ by~\cite[Corollary 3.2.2]{AD03}. Hence,  it follows that if $\lambda_G$ is GCR, then $G$ has a cocompact amenable subgroup.

In non-discrete setting, the GCR property for the group and its regular representation are not equivalent, see~\cite{Mackey61}. Examples of locally compact groups whose regular representation is GCR include all linear algebraic groups over a non-archimedean local field $k$ of characteristic zero~\cite{GooKal79}. 
\end{rmk}

We finish this section by mentioning an example of an amenable hyperbolic group that is not type~I. 

\begin{prop}\label{prop:non-type-I}
Let $F$ be a non-abelian finite simple group. Then the semi-restricted wreath product
$$G = (\bigoplus_{\mathbf Z_{< 0}} F) \oplus (\prod_{\mathbf Z_{\geq 0}} F) \rtimes \mathbf Z$$
is an amenable hyperbolic locally compact group that is not type~I.
\end{prop}
\begin{proof}
That $G$ is amenable and hyperbolic follows from~\cite[Theorem~A]{CCMT} (the amenability is obvious since $G$ is \{regionally elliptic\}-by-abelian). The regionally elliptic subgroup $G^0 = (\bigoplus_{\mathbf Z_{< 0}} F) \oplus (\prod_{\mathbf Z_{\geq 0}} F) $ has the compact group $K = \prod_{\mathbf Z_{\geq 0}} F$ as an open normal subgroup. The quotient $G^0/K$ is the discrete group $\bigoplus_{\mathbf Z_{< 0}} F$. The latter is not virtually abelian since $F$ is non-abelian simple. Hence it is not type~I by Thoma's theorem~\cite{Thoma}. Therefore $G^0$ is not type~I, hence $G$ is not type~I either, since the type~I property is inherited by open subgroups, see~\cite[Proposition~2.4]{Kallman73}. 
\end{proof}

\section{Conjectures and relation to $C^*$-simplicity}\label{sec:C*-simple}

In this section, we wish to relate Conjecture~\ref{conj:main} with another conjecture, pertaining to $C^*$-simplicity. Important recent results from~\cite{KK} ensure that if a discrete group $G$ has a (topologically) free action on its Furstenberg boundary, then its reduced $C^*$-algebra is simple. It is currently unknown whether the same result holds for non-discrete locally compact groups:

\begin{conj}\label{conj:C*-simple}
Let $G$ be a second countable locally compact group. If some point of the Furstenberg boundary of $G$ has a trivial stabilizer, then $C^*_r(G)$ is simple. 
\end{conj}

A few comments relevant to Conjecture~\ref{conj:C*-simple} may be found in the introduction of~\cite{CaLBMB}.
	
Known results in the discrete case make it natural to strengthen that conjecture as follows. Let $\mathcal A$ be the stabilizer URS associated with the $G$-action on its Furstenberg boundary. It is not difficult to see that $\mathcal A$  is the unique largest URS of $G$ consisting of relatively amenable subgroups. Moreover, that URS is reduced to the trivial subgroup if and only if some point of the Furstenberg boundary of $G$ has a trivial stabilizer. This is in turn equivalent to the condition that the Chabauty closure of the conjugacy class of each  relatively amenable closed subgroup of $G$ contains the trivial subgroup. 

\begin{conj}\label{conj:boundaryrep-C*-simple}
Let $G$ be a second countable  locally compact group and $\partial_F G$ be its Furstenberg boundary. For any continuity point $z \in \partial_F G$ of the stabilizer map $\partial_F G \to \Sub(G)$, the $C^*$-algebra $C^*_{\lambda_{G/G_z}}(G)$ is simple. 
\end{conj}

In other words,  using the notation of  Section~\ref{sec:Fell}, this conjecture predicts that $C^*(\mathcal A)$ is simple, where $\mathcal A$ is the stabilizer URS for the $G$-action on its Furstenberg boundary.

Conjecture~\ref{conj:boundaryrep-C*-simple} is known to hold if $G$ is discrete, see~\cite[Corollary~8.5]{Kawabe} and~\cite[Theorem~5.8]{KalSca20}. Moreover, that  conjecture is formally stronger than Conjecture~\ref{conj:C*-simple}: indeed, the latter is the special case of the former when $G_z$ is trivial (equivalently $\mathcal A = \{\langle e \rangle \}$). In the setting of hyperbolic groups, Conjecture~\ref{conj:boundaryrep-C*-simple} would imply that for any boundary representation $\pi$ of a non-amenable hyperbolic group $G$, the $C^*$-algebra $C^*_\pi(G)$ is simple, while Theorem~\ref{thm:weak-equiv-bd-rep} ensures that this algebra does not depend on the choice of $\pi$. 

The following result highlights a link between Conjecture~\ref{conj:boundaryrep-C*-simple} and Conjecture~\ref{conj:main}. 

\begin{prop}\label{prop:conj=>conj}
Let $G$ be a second countable tdlc  group satisfying the conclusion of Conjecture~\ref{conj:boundaryrep-C*-simple}. If $G$ is a type~I group, then it has a cocompact amenable subgroup. 
\end{prop}

The proof will be presented at the end of this section. The argument relies on the following two propositions of independent interest. 

\begin{prop}\label{prop:coco-amen}
For every  locally compact group $G$, the following assertions are equivalent. 
\begin{enumerate}[label=(\roman*)]
\item $G$ has a cocompact amenable subgroup. 

\item $G$ acts transitively on its Furstenberg boundary.
\end{enumerate}
\end{prop}
\begin{proof}
Suppose that the topological group $G$ has a homogeneous Furstenberg boundary $\partial_F G$ and write $\partial_F G \cong G/P$ for a subgroup $P<G$ which is necessarily cocompact. The universal property of $\partial_F G$ shows that $P$ is relatively amenable in $G$, but by cocompactness it is actually amenable. This latter point is proved e.g.\ in the Proposition in~\cite{Monod_Gelfand}.

Conversely, suppose that $G$ contains a cocompact amenable subgroup $P$. Since $P$ must fix a point in the space $\Delta G$ of probability measures on $\partial_F G$, it follows that $G$ has a compact orbit in $\Delta G$. By the converse to Krein--Milman's theorem~\cite[V.8.5]{Dunford-Schwartz_I}, this orbit contains all extremal points of $\Delta G$ and hence $\partial_F G$ is a single orbit (see~\cite{Glasner} for details).
\end{proof}

\begin{prop}\label{prop:transitive-boundary}
Let $G$ be a $\sigma$-compact locally compact group, let $Z$ be a $G$-boundary, let $\nu$ be a $G$-quasi-invariant Radon probability measure on $Z$ such that $L^2(Z, \nu)$ is separable. Let $\kappa$ denote the associated Koopman representation, and let $U \leq G$ be a compact open subgroup that fixes the measure $\nu$. If the image of the convolution algebra $C_c(U\backslash G/U)$ under $\kappa$ is finite-dimensional, then the $G$-action on $Z$ is transitive. 
\end{prop}

We need the following elementary fact.  

\begin{lem}\label{lem:double-coset}
Let $G$ be a locally compact group with left Haar measure $\mu$, let $U \leq G$ be a compact open subgroup and  $\pi$ be a unitary representation of $G$. For each $g \in G$, the characteristic function $\mathbf 1_{UgU}$ belongs to $C_c(G)$ and we have
$$\pi(\one_{UgU}) = \mu(UgU) p_U \pi(g) p_U,$$
where $p_U $ is the orthogonal projection on the subspace of $U$-invariant vectors in $\mathcal H_\pi$. 
\end{lem}

\begin{proof}
The identity is unaffected by renormalizing $\mu$, so we can assume $\mu(U)=1$. The integral expression $p_U =  \pi(\one_{U})$ for $p_U$ is well-known and easy to see. It follows that the right hand side of the identity to establish is $\mu(UgU) \pi(\one_U * \one_{g U})$ and it suffices to check the convolution equation
$
\one_{UgU} = \mu(UgU) \cdot \one_U * \one_{g U}.
$
By definition, $(\one_U * \one_{g U})(z) = \int_U \one_{g U} (y^{-1} z) \, d\mu(y) = \mu(U\cap z U g^{-1})$. Since $U$ is compact open, it is partitioned into $m\geq 1$ cosets $u_i (U \cap g U g^{-1})$ of measure $1/m$, with $u_i\in U$. Writing $U = \bigsqcup_{i=1}^m U \cap u_i g U g^{-1}$, it follows that $\mu(U\cap z U g^{-1})$ is~$1/m$ or~$0$ according to whether $z\in UgU$ or not. This implies
\[
\one_{UgU} = m \cdot \one_U * \one_{g U}.\eqno{(*)}
\]
It remains to check $m=\mu(UgU)$, which can be derived from integrating both sides of~$(*)$. 
\end{proof}

\begin{proof}[Proof of Proposition~\ref{prop:transitive-boundary}]
Suppose that $\kappa(C_c(U\backslash G/U))$ is finite-dimensional. Then there exist elements $g_0, g_1, \dots, g_k \in G$ such that $\kappa(C_c(U\backslash G/U))$ is spanned by the images of the characteristic functions $\mathbf 1_{Ug_0 U}$, \dots, $\mathbf 1_{Ug_k U}$. Suppose for a contradiction that there exist points $x, y \in Z$ in distinct $G$-orbits. Since $U$ is compact, it follows that there exists a compact neighbourhood $\alpha$ of $x$ such that $y \not \in \bigcup_{i=0}^k Ug_i U \alpha$. Since the set $\bigcup_{i=0}^k Ug_i U \alpha$ is closed, there exists also a neighbourhood $\beta$ of $y$ with $\beta \cap (\bigcup_{i=0}^k Ug_i U \alpha) = \varnothing$. In particular $\beta \cap  Ug_i U \alpha = \varnothing$ for all $i$, so that 
$$\langle \kappa(\mathbf 1_{Ug_i U})\mathbf 1_\alpha, \mathbf 1_\beta \rangle = 0$$
for all $i$, in the space $L^2(Z, \nu)$. 

Choose neighbourhoods  $\alpha' \subsetneq \alpha$ of $x$ and $\beta' \subsetneq \beta$ of $y$ such that  $U' := \{g \in U \mid g(\alpha') \subseteq \alpha, \  g(\beta') \subseteq \beta\}$ is an identity neighbourhood in $G$. Upon replacing $U'$ by $U' \cap (U')^{-1}$ we shall assume that $U'$ is symmetric. 

Since $Z$ is a $G$-boundary, there exists $t \in G$ such that $\nu(t\alpha') > 1-\nu(\beta')$. By the definition of $g_0, \dots, g_k$, there exist $\lambda_0, \dots, \lambda_k \in \mathbf C$ such that $\kappa(\mathbf 1_{U t U}) = \sum_{i=0}^k \lambda_i \kappa(\mathbf 1_{Ug_i U})$. By the above, it follows that 
$\langle \kappa(\mathbf 1_{U t U})\mathbf 1_\alpha, \mathbf 1_\beta \rangle = 0.$
Using Lemma~\ref{lem:double-coset}, we deduce that 
\begin{align*}
0 & = \langle p_U \kappa(t) p_U\mathbf 1_\alpha, \mathbf 1_\beta \rangle\\
& = \langle \kappa(t) p_U\mathbf 1_\alpha, p_U \mathbf 1_\beta \rangle\\
& = \int_Z (p_U \mathbf 1_\alpha)(t^{-1}z) (p_U \mathbf 1_\beta)(z) \sqrt{\frac{dt\nu}{d\nu}(z)} d\nu(z).
\end{align*}
Therefore the non-negative map $z \mapsto p_U \mathbf 1_\alpha(t^{-1} z) p_U \mathbf 1_\beta(z) \sqrt{\frac{dt\nu}{d\nu}(z)} $ vanishes $\nu$-almost everywhere. 
Since $t\nu$ is in the same measure class as $\nu$, the Radon-Nikodym derivative $d t\nu/d\nu$ is non-zero $\nu$-a.e. Therefore, the map  $z \mapsto p_U \mathbf 1_\alpha(t^{-1} z) p_U \mathbf 1_\beta(z)$ also vanishes $\nu$-a.e.
Moreover, by Lemma~\ref{lem:double-coset} we have 
\begin{align*}
p_U \mathbf 1_\beta(z) & = \frac 1 {\mu(U)} \int_U \mathbf 1_\beta(u^{-1} z) d\mu(u) \\
& \geq  \frac 1 {\mu(U)} \int_{U'} \mathbf 1_\beta(u^{-1} z) d\mu(u)\\
& \geq  \frac 1 {\mu(U)} \int_{U'} \mathbf 1_{\beta'}( z) d\mu(u)\\
& = \frac {\mu(U')} {\mu(U)}\mathbf 1_{\beta'}( z).
\end{align*}
Similarly, we have $p_U \mathbf 1_\alpha(z) \geq  \frac {\mu(U')} {\mu(U)}\mathbf 1_{\alpha'}( z)$.
 It follows that 
\begin{align*}
0 & =\int_Z (p_U \mathbf 1_\alpha)(t^{-1}z) (p_U \mathbf 1_\beta)(z) d\nu(z)\\
& \geq  \frac {\mu(U')^2} {\mu(U)^2} \int_Z  \mathbf 1_{\alpha'}(t^{-1}z) \mathbf 1_{\beta'}(z) d\nu(z)\\
& =  \frac {\mu(U')^2} {\mu(U)^2} \nu(t \alpha' \cap \beta').
\end{align*}
Therefore $\nu(t \alpha' \cap \beta') = 0$ so that $\nu(t \alpha') \leq 1 - \nu(\beta')$. This contradicts the definition of $t$. 
\end{proof}

\begin{proof}[Proof of Proposition~\ref{prop:conj=>conj}]
Let $G$ be a second countable tdlc group of type~I and $Z$ be its  Furstenberg boundary. Let $z \in Z$ be a continuity point of the stabilizer map $Z \to \Sub(G)$, so that $G_z$ belongs to the URS  $\mathcal A$. Let $\nu_0$ be a $G$-quasi-invariant Radon probability measure on $G/G_z$, and $\nu$ denote the push-forward of $\nu_0$ along the orbit map $G/G_z \to Z$. Then $L^2(Z, \nu)$ is separable since $L^2(G/G_z, \nu_0)$ is, and by construction the Koopman representation $\kappa$ on $L^2(Z, \nu)$ is unitarily equivalent to the quasi-regular representation $\lambda_{G/G_z}$. Let $U \leq G$ be a compact open subgroup. By Lemma~\ref{lem:conv->equiv-meas}, there exists a $U$-invariant probability measure $\nu'$ in the measure-class of $\nu$. In view of Lemma~\ref{lem:equiv-meas->equiv-koop}, we may replace $\nu$ by $\nu'$ and assume henceforth that  $\nu$  is $U$-invariant. 

By construction, we have $C^*_\kappa(G) \cong C^*_{\lambda_{G/G_z}}(G)  = C^*(\mathcal A)$. It follows by hypothesis that $C^*_\kappa(G)$ is simple. Since moreover $G$ is a type~I group by assumption, we infer that $\kappa$ is CCR. Let $\pi$ be an irreducible representation of $C^*_\kappa(G)$ (considered also as a unitary representation of $G$).
Then since the operator $\pi(\mathbf 1_U)$ is compact, $\pi\big(C_c(U \backslash G / U)\big)$ is finite-dimensional. Therefore $\kappa\big(C_c(U \backslash G / U)\big)$ is finite-dimensional since $C^*_\pi(G) = C^*_\kappa(G)$.
This ensures that all the hypotheses of Proposition~\ref{prop:transitive-boundary} are satisfied. Hence $G$ acts transitively on $Z$, so that $G$ has a cocompact amenable subgroup by Proposition~\ref{prop:coco-amen}. 
\end{proof}

\begin{rmk}
Proposition~\ref{prop:conj=>conj} does not require any hypothesis of irreducibility of boundary representations. Hence, for a hyperbolic group $G$ satisfying the conclusion of Conjecture~\ref{conj:boundaryrep-C*-simple}, it provides a conceptually different approach to the existence of a cocompact amenable subgroup, in comparison with the proof Theorem~\ref{thm:main}. 
\end{rmk}

We conclude this discussion by showing that, in order to establish Conjecture~\ref{conj:main}, it suffices to prove it in the special case of unimodular tdlc groups. We first record the following observation. 

\begin{prop}\label{prop:open-normal}
Let $Q_0\lhd Q$ be an open normal subgroup of a locally compact group $Q$ and assume that $Q/Q_0$ is amenable. If $Q_0$ admits a cocompact amenable subgroup, then so does $Q$.
\end{prop}

\begin{proof}
According to (the proof of) \Cref{prop:coco-amen}, the Furstenberg boundary of $Q_0$ is of the form $Q_0/P_0$ for a cocompact amenable subgroup $P_0<Q_0$. Moreover, the action of $Q_0$ on its Furstenberg boundary can be extended to a  $Q$-action by homeomorphisms, see Proposition~II.4.3 and page~32 in~\cite{Glasner}. This $Q$-action is continuous because $Q_0$ is open in $Q$. It is moreover transitive since already $Q_0$ is transitive. Therefore, $Q$ admits a cocompact subgroup $P<Q$ with $P\cap Q_0 = P_0$. This subgroup is amenable since it is an extension of $P_0$ by a subgroup of the amenable discrete group $Q/Q_0$.
\end{proof}

\begin{prop}\label{prop:tdlc-reduction}
If every (second countable)  unimodular tdlc group satisfies the conclusion of \Cref{conj:main}, then every  (second countable) locally compact group does as well. 
\end{prop}

\begin{proof}
Let $G$ be an arbitrary (second countable) locally compact group of type~I. Let $R\lhd G$ be its {amenable radical}, that is, the largest normal (topologically) amenable subgroup, which is automatically closed and unique. It follows from structure theory and Lie theory that $G/R$ admits a finite index open subgroup $H<G/R$ which splits as a direct product $H=S\times Q$, where $S$ is a connected semi-simple Lie group and $Q$ is totally disconnected. This is proved as Theorem~3.3.3 in~\cite{Burger-Monod3} (see also~\cite[11.3.4]{Monod}). Finally, let $Q_0$ be the kernel of the modular function of $Q$.

Recall first that $Q_0$ is unimodular: it is actually the maximal unimodular closed normal subgroup of $Q$, see~\cite[VII \S 2 No 7]{NBourbakiINT78}. Moreover $Q_0$ is second countable if $G$ is so. Observe next that $Q_0$ is open in $Q$ because it contains every compact subgroup of $Q$ and there are open compact subgroups by van Dantzig's theorem. Finally, we claim that $Q_0$ is of type~I. Indeed, the type~I condition passes to quotients by definition, and is also passes to open subgroups, see e.g.\ Proposition~2.4 in~\cite{Kallman73} (this reference assumes second countability but it is not used in~2.4). Applying each of these hereditary properties twice, we can pass from $G$ to $Q_0$.

By hypothesis, the group $Q_0$ satisfies the conclusion of \Cref{conj:main}, hence it contains  a  cocompact amenable subgroup $P_0$. In view of Proposition~\ref{prop:open-normal}, we now know that $Q$ admits a cocompact amenable subgroup $P_Q< Q$. On the other hand, $S$ also admits a cocompact amenable subgroup $P_S< S$, namely a minimal parabolic subgroup. The amenable group $P_Q \times P_S$ is still cocompact in $G/R$ since $S\times Q$ has finite index in the latter. It follows that the pre-image $P<G$ of $P_Q \times P_S$ in $G$ is cocompact in $G$. On the other hand, $P$ is amenable since it is an extension of $R$ by $P_Q \times P_S$.
\end{proof}

\begin{rmk}\label{rem:tdlc-reduction}
Combining Propositions~\ref{prop:conj=>conj} and~\ref{prop:tdlc-reduction}, we see that if Conjecture~\ref{conj:boundaryrep-C*-simple} holds for all second countable unimodular tdlc groups, then Conjecture~\ref{conj:main} is true. 
\end{rmk}

\subsection*{Acknowledgements}
The first named author is grateful to Adrien Le Boudec for stimulating discussions around Conjecture~\ref{conj:C*-simple}. We thank Uri Bader, Bachir Bekka, Pierre de la Harpe, Cyril Houdayer, Basile Morando and Sven Raum for their comments on a preliminary version of this paper.


\bibliographystyle{amsalpha}
\bibliography{biblio}

\end{document}